\documentclass[11pt, a4paper]{amsart}
\usepackage{amssymb, amsmath, amsthm, amsfonts, mathtools, multicol, inputenc}

\usepackage{fancyhdr}
\usepackage[a4paper,left=25mm,right=25mm,top=30mm,bottom=30mm]{geometry}

\usepackage{mathrsfs}
\usepackage{graphicx}
\usepackage{color}
\usepackage[normalem]{ulem}
\usepackage{cancel}
\usepackage{tikz}
\usetikzlibrary{matrix}
\usepackage[all]{xy}
\usepackage{ytableau,varwidth}
\usepackage{pgfplots,caption}

\newtheorem{theorem}{Theorem}[section]
\newtheorem{corollary}[theorem]{Corollary}
\newtheorem{lemma}[theorem]{Lemma}
\newtheorem{proposition}[theorem]{Proposition}

\newtheorem{conjecture}[theorem]{Conjecture}

\theoremstyle{definition}
\newtheorem{example}[theorem]{Example}

\theoremstyle{definition}

\newtheorem{remark}[theorem]{Remark}

\usepackage[normalem]{ulem}

\newcommand{\OO}{\mathcal O}
\newcommand{\EE}{\mathcal E}
\newcommand{\eps}{\varepsilon}
\newcommand{\Z}{\mathbb Z}
\newcommand{\Q}{\mathbb Q}

\newcommand{\N}{\mathbb N}
\newcommand{\F}{\mathbb F}
\newcommand{\Hom}{\operatorname{Hom}}
\newcommand{\End}{\operatorname{End}}

\newcommand{\GL}{\operatorname{GL}}

\newcommand{\Ker}{\operatorname{Ker}}

\newcommand{\Ext}{{\operatorname{Ext}}}

\newcommand{\Gal}{\operatorname{Gal}}

\newcommand{\Irr}{\operatorname{Irr}}
\newcommand{\IBr}{\operatorname{IBr}}

\newcommand{\opp}{\operatorname{op}}
\newcommand{\rank}{\operatorname{rank}}

\newcommand{\Res}{{\operatorname{Res}}}

\newcommand{\PSL}{{\operatorname{PSL}}}
\newcommand{\U}{\mathcal{U}}
\newcommand{\V}{\mathrm{V}}

\thanks{The second author acknowledges financial support by the Spanish Ministry of Science and Innovation through the Severo Ochoa Programme Grant CEX2019-000904-S funded by MCIN/AEI/10.13039/501100011033. The first author was partially supported by a travel grant by the DKO trust. A visit by the second author to Manchester was funded by EPSRC grant EP/T004606/1.}

\title{Units in Blocks of Defect 1 and the Zassenhaus Conjecture}

\author{Florian Eisele}
\address{Department of Mathematics, University of Manchester, Oxford Road, Manchester, M13 9PL}
\email{florian.eisele@manchester.ac.uk}

\author{Leo Margolis}
\address{ICMAT, C/ Nicolas Cabrera 13, 28049 Madrid, Spain.}
\email{leo.margolis@icmat.es}

\renewcommand{\leq}{\leqslant}
\renewcommand{\geq}{\geqslant}

\usepackage[hidelinks]{hyperref}
\usepackage{stix}

\usepackage{geometry}

\begin{document}

\maketitle

\begin{abstract}
   Building on previous work by Caicedo and the second author, we develop a method that decides the existence of units of finite order in blocks of $\Z_p G$ of defect 1. This allows us to prove that if $p$ is a prime and $G$ is a finite group whose Sylow $p$-subgroup has order $p$, then any unit  $u\in \Z G$ of order $p$ is conjugate to an element of $\pm G$. This is a special case of the Zassenhaus conjecture. We also prove some new results on units of finite order in $\Z \PSL(2,q)$ for certain $q$, and construct a unit of order $15$ in $V(\Z_{(3,5)}\PSL(2,16))$ which is a $3$- and $5$-local counterexample to the Zassenhaus conjecture, raising the hope that our methods may lead to a global counterexample amongst simple groups. 
\end{abstract}

\section{Introduction}

One of the fundamental questions about the group ring $\Z G$ for a finite group $G$ is what the finite subgroups of its unit group look like. In this paper, we are interested in individual units, or, equivalently, finite cyclic subgroups of the unit group $\U(\Z G)$. Zassenhaus conjectured in \cite{Zassenhaus} that 

\begin{conjecture}[Zassenhaus]\label{conj:zassenhaus}
   If $u$ is a unit of finite order in $\Z G$, then $u$ is conjugate within $\Q G$ to $\pm g$ for some $g\in G$.
\end{conjecture}

While this conjecture is false in general \cite{EiseleMargolis}, it is known to hold in many classes of groups, cf. \cite[Section 4]{MargolisdelRioSurvey} for an overview. In particular there are no known counterexamples to this conjecture where the unit $u$ has $p$-power order for a prime $p$. The question whether any $p$-subgropup of $\U (\Z G)$ is conjugate within $\Q G$ to a subgroup of $\pm G$ remains open as well. In this paper we prove a positive result in this direction. Using the theory of blocks of cyclic defect we are able to prove the following theorem.

\begin{theorem}\label{thm:main}
   Let $G$ be a finite group and let $p$ be a prime. If the Sylow $p$-subgroup of $G$ has order $p$, then any unit $u\in \U(\Z G)$ of order $p$ is conjugate within $\Q G$ to $\pm g$ for some $g\in G$.
\end{theorem}

We prove this result using the \emph{lattice method} developed in \cite{BachleMargolisLattice, BachleMargolis4primaryII}, combined with the theory of blocks of defect 1. This method considers the restrictions of various irreducible $\OO G$-lattices, for a suitable complete discrete valuation ring $\OO$, to $\OO\langle u \rangle$, where $u$ is assumed to be a counterexample to the theorem, and then places various constraints on the reductions of those lattices modulo the uniformizr of $\OO$. This requires understanding the representation theory of a cyclic group of order $p$ over $\OO$, which is known to be wild if the ramification index of $\OO$ is bigger than $2$ (barring a few exceptions). While we cannot avoid considering rings $\OO$ over which the representation theory is wild, it turns out that the particular question which is relevant in the context of the lattice method can be answered regardless (see Theorem~\ref{thm:filtration}). This allows us to expand on the results of \cite{CaicedoMargolis} to prove the theorem above.

While there are many positive results on Conjecture~\ref{conj:zassenhaus} for solvable groups, the situation is much more bleak when it comes to non-solvable groups and, in particular, non-abelian simple groups. Here it is only known to hold for groups of type $\operatorname{PSL}(2,q)$ when $q \leq 32$ or $q$ is a Fermat or  Mersenne prime \cite{MargolisdelRioSerrano}. There are two problems when trying to generalize this to all groups of type $\operatorname{PSL}(2,q)$: units of order divisible by the defining characteristic when $q$ is at least the third power of a prime, and  units of order $2t$ for $t$ a prime different from the defining characteristic. In the latter case the known methods were shown to fail in \cite{delRioSerrano17}. We can overcome this difficulty under the assumption that $t^2$ does not divide the order of the group, cf. Theorem~\ref{th:PSL2Ord2t}. This allows us to prove:

\begin{corollary}[{see Corollary~\ref{cor:zassenhauspsl}}]\label{cor:zassenhauspslIntro}
Let $p$ be a prime and set $q=p$ or $q=p^2$. Assume that there exists a prime $t$ so that $q-1 = 4t$ or $q+1 = 4t$. Then the Zassenhaus conjecture holds for $\operatorname{PSL}(2,q)$.
\end{corollary}

It is part of a number-theoretical conjecture that there are infinitely many primes satisfying the condition of Corollary~\ref{cor:zassenhauspslIntro}.

Apart from these positive results, we also show how the search for counterexamples to Conjecture~\ref{conj:zassenhaus} in the $p$-local case can be reduced to a purely combinatorial problem when $p^2$ does not divide the order of the group $G$. More generally, this applies to the construction of units of order divisible by $p$ in blocks of $\Z_p G$ of defect $1$  (asking that $p^2$ does not divide $|G|$ simply ensures that there are no blocks of higher defect). The structure of a block of defect 1 is decribed by its \emph{Brauer tree}. 
To construct a unit of order $pr$ in such a block, where $r$ is co-prime to $p$, it suffices to attach a tuple of Young diagrams to each vertex and each edge of the Brauer tree. This corresponds to assigning $\bar \F_p C_{pr}$-modules to all vertices and edges. One then has to check that each $\bar \F_p C_{pr}$-module attached to a vertex of the Brauer tree is filtered in a particular way by the modules attached to the adjacent edges. One also has to check that there is a second filtration determined by the multiplicities of the different eigenvalues of the unit on the ordinary representation corresponding to the vertex.
Checking the existence of these two filtrations is an entirely combinatorial procedure. The detailed conditions are given in Theorem~\ref{thm:reversaloflatticemethod}. We hope that this procedure will eventually lead to new counterexamples to Conjecture~\ref{conj:zassenhaus}, in particular in non-solvable groups (see Proposition~\ref{prop:CliffWeissApp}). In Example~\ref{ex:PSL216} we provide a unit $u\in \V(\Z_{(3,5)}\PSL(2,16))$ of order $15$ which is a $3$- and $5$-local counterexample to Conjecture~\ref{conj:zassenhaus}. It is known that such a unit cannot exist globally, i.e. in $\V(\Z G)$, but curiously the obstruction is $2$-local even though $2$ does not divide the order of $u$.

%Another question, related to Conjecture~\ref{conj:zassenhaus}, is the \emph{Prime Graph Question}. It asks whether there is a normalized unit of order $p\cdot q$ in $\Z G$ (see section~\ref{section:torsionunits})  if and only if there is an element of order $p\cdot q$ in $G$, where $p\neq q$ are two arbitrary primes. We show that Conjecture~\ref{conj:zassenhaus} holds for the groups $\PSL(2,q)$, where $q=p$ or $q=p^2$ and $q-1=4t$ or $q+1=4t$ for some prime $t$ (see Corollary~\ref{cor:zassenhauspsl}). This is particularly interesting in the context of the Prime Graph Question, since it has a reduction to almost simple groups and, unlike Conjecture~\ref{conj:zassenhaus}, there are no known counterexamples. 

\section{Basic facts and notation}

In this section we summarize a few basic facts from  various subject areas we will need later. Throughout the whole article $p$ denotes a prime, $G$ a finite group and $R$ a commutative ring with identity. The units of a ring $S$ are denotes as $\U(S)$. For a field $K$ we denote by $\bar{K}$ the algebraic closure of $K$. We denote by $\mathbb{Z}_p$ the $p$-adic integers and by $\mathbb{Z}_{(p)}$ the localization of $\mathbb{Z}$ at $p$. Complete discrete valuation rings are always assumed to be of characteristic $0$. If $\mathcal{O}$ is a complete discrete valuation ring with residue field of characteristic $p>0$, then we define the \emph{ramification index of $\mathcal{O}$} as the ramification index of the ideal $p\mathcal{O}$ with respect to the extension $\mathcal{O}/\mathbb{Z}_p$. Moreover, for a character $\chi$ we denote by $\mathcal{O}[\chi]$ the smallest ring containing $\mathcal{O}$ and the values of $\chi$.

\subsection{Torsion units in integral group rings}\label{section:torsionunits}
For a group ring $RG$ we denote by
\begin{align*}
\varepsilon : &\  RG \longrightarrow R \\
& \sum_{g \in G} r_gg \mapsto \sum_{g \in G} r_g
\end{align*}
the \emph{augmentation map}. The subgroup of $\U(RG)$ consisting of units of augmentation $1$ is called the group of \emph{normalized units} and denoted as $\V(RG)$. It is easy to see that $\U(RG) \cong \U(R) \times \V(RG)$, so that for most questions on units of group rings it suffices to consider normalized units. The Zassenhaus conjecture is then equivalent to the claim that each unit $u$ of finite order in $\V(\mathbb{Z}G)$ is conjugate in $\U(\mathbb{Q}G)$ to an element of $G$. This fact is also expressed by saying that $u$ is \emph{rationally conjugate} to an element of $G$.

For an element $x \in G$ denote by $x^G$ the conjugacy class of $x$ in $G$. We denote by 
\begin{align*}
\varepsilon_x: &RG \rightarrow R \\
& \sum_{g \in G} r_gg \mapsto \sum_{g \in x^G} r_g
\end{align*}
the \emph{partial augmentation} at $x$. Partial augmentations are a fundamental tool in the study of torsion units of integral group rings and we cite the basic results about them that we will need later:

\begin{theorem}\cite[Theorem 2.5]{MarciniakRitterSehgalWeiss}\label{th:MRSW}
Let $u \in \V(\mathbb{Z}G)$ be a unit of finite order $n$. Then $u$ is rationally conjugate to an element of $G$ if and only if $\varepsilon_x(u^d) \geq 0$ for all $x \in G$ and all divisors $d$ of $n$.
\end{theorem}

\begin{theorem}\label{th:pAsofTorsionUnits}
Let $u \in \V(\mathbb{Z}G)$ be of order $n$.
\begin{itemize}
\item[(i)] If $u \neq 1$, then $\varepsilon_1(u) = 0$ (Berman-Higman Theorem) \cite[Proposition 1.5.1]{GRG1}.
\item[(ii)] Let $g\in G$ such that $\varepsilon_g(u) \neq 0$. Then the order of $g$ divides $n$ \cite[Proposition 2.2]{HertweckBrauer}.
\end{itemize}
\end{theorem}

%\begin{theorem}\label{th:pAsofTorsionUnits}[Hertweck] Let $u \in V(\mathbb{Z}G)$ be of order $n$.
%\begin{itemize}
%\item If $\varepsilon_x(u) \neq 0$, then the order of $x$ divides $n$.
%\item If the $p$-part of $u$ is conjugate to $y \in G$ in $V(\mathbb{Z}_pG)$ and $\varepsilon_x(u) \neq 0$, then the $p$-part of $x$ is conjugate in $G$ to $y$.
%\item If $N$ is a normal $p$-subgroup of $G$ and $u$ maps to the identity under the homomorphism $\mathbb{Z}G \rightarrow \mathbb{Z}(G/N)$, then $u$ is conjugate in $V(\mathbb{Z}_pG)$ to an element of $G$.
%\end{itemize}
%\end{theorem}

%\begin{lemma}\label{lem:DokuchaevJuriaans}[Dokuchaev-Juriaans]
%Let $u \in V(\mathbb{Z}G)$ be of order $n$ and $N$ be a normal subgroup of $G$ of order relatively prime to $n$. Denote $\varphi: \mathbb{Z}G \rightarrow \mathbb{Z}(G/N)$. Then $\varepsilon_g(u) = \varepsilon_{\varphi(g)}(\varphi(u))$.
%\end{lemma}

Let $\ell$ be a positive integer and $K$ a field of characteristic $0$. Denote by $M_\ell(K)$ the ring of $\ell \times \ell$-matrices over $K$ and by $\operatorname{GL}_\ell(K)$ the group of invertible elements of $M_\ell(K)$. Consider a group representation $D: G \rightarrow \operatorname{GL}_\ell(K)$ with associated character $\chi$. Then $D$ naturally extends to a ring homomorphism $\mathbb{Z}G \rightarrow M_\ell(K)$ which in turn restricts to a group homomorphism $\U(\mathbb{Z}G) \rightarrow \operatorname{GL}_\ell(K)$. This allows us to extend $\chi$ to $\U(\mathbb{Z}G)$ in a linear way so that for $u \in \U(\mathbb{Z}G)$ we have
\begin{equation}\label{eq:CharactersAndPartAugs}
\chi(u) = \sum_{g^G} \varepsilon_g(u) \chi(g),
\end{equation}
where ``$\sum_{g^G}$'' denotes a sum over representatives for the conjugacy classes of $G$. If $u$ has order $n$, the matrix $D(u)$ has order divisible by $n$ so that it is diagonalizable over $\bar{K}$ and all the eigenvalues of $D(u)$ are $n$-th roots of unity. For $\zeta$ any $n$-th root of unity we denote by $\mu(\zeta,u,\chi)$ the multiplicity of $\zeta$ as an eigenvalue of $D(u)$. We use the standard notation ${\rm Tr}_{L/K}$ for the trace map on an algebraic field extension $L/K$. The following is essentially a reformulation of the second orthogonality relation of characters which is very useful to calculate multiplicities of eigenvalues.

\begin{proposition}\cite{LutharPassi1989}\label{pr:luthar-passi-multiplicity-formula}
   Let $u$ in $\V(\mathbb Z G)$ be of order $n$,  $\zeta\in \mathbb C$ an $n$-th root of unity and $\chi$ an ordinary character of $G$. Then 
   \[
       \mu(\zeta, u, \chi) =\frac{1}{n}\sum_{d|n} {\rm Tr}_{\mathbb Q(\zeta^d)/\mathbb Q} (\chi(u^d)\zeta^{-d})
   \] 
\end{proposition}
Note that \eqref{eq:CharactersAndPartAugs} allows us to replace the values $\chi(u^d)$ in Proposition~\ref{pr:luthar-passi-multiplicity-formula} with expressions only involving the partial augmentations of powers of $u$ and the values of $\chi$ at elements of $G$.

\subsection{Modules of cyclic groups and combinatorics}
We refer to \cite{CaicedoMargolis} for precise references for all facts mentioned in this section. Throughout this section we fix a prime $p$ and denote by $k$ a field of characteristic $p$ and by $\mathcal{O}$ a complete discrete valuation ring with residue field of characteristic $p$. 
We first describe the general module theory for cyclic groups over $k$ and $\mathcal{O}$. Throughout this paper, all modules are assumed to be finitely generated.

\begin{proposition}\label{prop:kCpModules}
Let $n$ be a positive integer, $C$ a cyclic group of order $p^n$ and $M$ an indecomposable $kC$-module. Then $M$ is uniserial of dimension at most $p^n$ and the dimension of $M$ characterizes $M$ up to isomorphism. Moreover, for any positive integer $\ell$ with $\ell \leq p^n$ there is an indecomposable $kC$-module of dimension $\ell$. 
\end{proposition}

\subsection*{Notation} For a cyclic $p$-group $C$ and $\ell$ a positive integer such that $\ell \leq |C|$ we denote by $I_\ell$ the indecomposable $kC$-module of dimension $\ell$ which by Proposition~\ref{prop:kCpModules} is a well-defined object.

%The situation is much more complicated for lattices of cyclic groups over $\mathcal{O}$ and indeed one of our main results is a contribution to the understanding of it. In some special cases though it is well-controlled and the most relevant for us is the following.
%
%\begin{proposition}\label{prop:UnramifiedLattices}
%Assume $p$ is unramified in $\mathcal{O}$ and $L$ an indecomposable $\mathcal{O}C_p$-lattice. Then $L$ is of rank $1$, $p-1$ or $p$ and isomorphic to the trivial module, the augmentation ideal of $\mathcal{O}C_p$ or $\mathcal{O}C_p$ respectively. $L$ remains indecomposable when passing from $\mathcal{O}$ to the residue field of $\mathcal{O}$.
%\end{proposition}
%\ChLeo{Nachdem ich das geschrieben habe, ist mir aufgefallen, dass man es auch als Korollar deines Hauptsatzes reinnehmen koennte.}

\begin{proposition}\label{prop:CyclicGroupModulesp'Decomposition}
Let $C = \langle g \rangle$ be cyclic group of order $p^n\ell$ where $\ell$ is an integer not divisible by $p$. Assume $R = k$ or $R = \mathcal{O}$ such that $R$ contains a primitive $\ell$-th root of unity $\xi$. Then any $RC$-module $M$ decomposes as 
\[M = M_1 \oplus M_2 \oplus ... \oplus M_\ell \]
where $M_i$ is the maximal $R$-subspace of $M$ on which the $p'$-part of $g$ acts as $\xi^i$.  
\end{proposition}

\subsection*{Notation} Propositions~\ref{prop:kCpModules} and \ref{prop:CyclicGroupModulesp'Decomposition} justify the following notation. For a cyclic group $C = \langle g \rangle$ of order $p^n\ell$ with $p \nmid \ell$, a field $k$ of characteristic $p$ containing a primitive $\ell$-th root of unity $\xi$ and a $kC$-module $M$ we denote by $\gamma_{j,i}(M)$ the number of indecomposable direct summands of $M$ of dimension at least $j$ on which the $p'$-part of $g$ acts as $\xi^i$.

We next introduce the combinatorics needed to understand extensions of modules of cyclic $p$-groups over $k$. We will be working with Young tableaux and using cardinal directions to express locations inside these tableaux, e.g. ``northwest'' means upper left. Let $\lambda = (\lambda_1,\ldots,\lambda_r)$ be a partition and $\mu = (\mu_1,\ldots,\mu_s)$ a subpartition of $\lambda$, i.e. $s \leq r$ and $\mu_i \leq \lambda_i$ for any $i$. The \emph{skew diagram} associated to $\lambda/\mu$ is a Young diagram of shape $\lambda$ where a Young diagram of shape $\mu$ has been removed in the northwestern corner of the Young diagram of $\lambda$. A skew diagram is a \emph{skew tableau} if each of its boxes contains a positive integer as an entry. Let $T$ be a skew tableau. $T$ is called \emph{semi-standard} if entries are non-decreasing in each row, reading left-to-right, and entries are strictly increasing in each column reading top-to-bottom. $T$ is said to satisfy the \emph{lattice property} if for any box $b$ in $T$ reading the word $w$ from $T$ right-to-left and top-to-bottom the word we obtain when reading up until $b$ contains at least as many 1's as 2's, 2's as 3's, etc. Now let $w$ be the word obtained by reading all of $T$ is this manner, from the northeastern to the southwestern corner, and let $\nu_i$ be the number of times the letter $i$ appears in $w$. If $T$ satisfies the lattice property and $t$ is the maximal number appearing as a letter in $w$, then $\nu = (\nu_1,\ldots,\nu_t)$ is a partition of the number of boxes in $T$ known as the \emph{content} of $T$.

Now let $C$ be a cyclic group of order $p^n$ and $M$ a $kC$-module. By Proposition~\ref{prop:kCpModules} the isomorphism type of $M$ is determined by the dimensions of its indecomposable summands. If we order these dimensions in a non-increasing way, including multiplicities, we obtain a tuple $\lambda = (\lambda_1,\ldots,\lambda_r)$ which is a partition of the dimension of $M$ and each $\lambda_i \leq p^n$. We call $\lambda$ the \emph{partition associated to $M$}. The extension theory of $kC$-modules is then explained by the following combinatorial fact.

\begin{theorem}\label{th:ModulesAndTableaux}
Let $C$ be a cyclic group of order $p^n$ and let $M$, $U$ and $Q$ be $kC$-modules with associated partitions $\lambda$, $\mu$ and $\nu$, respectively. Then $M$ has a submodule $U'$ isomorphic to $U$ such that $M/U' \cong Q$ if and only if there exists a semi-standard skew tableau satisfying the lattice property of shape $\lambda/\mu$ and content $\nu$.
\end{theorem}

The following combinatorial lemma will simplify our representation theoretic problems.

\begin{lemma}\label{lem:SymmetryOfLRCoef}
Let $\lambda$, $\mu$ and $\nu$ be partitions. Then there is a semi-standard skew tableau satisfying the lattice property of shape $\lambda/\mu$ and content $\nu$ if and only if there is a semi-standard skew tableau satisfying the lattice property of shape $\lambda/\nu$ and content $\mu$.
\end{lemma}

Since Theorem~\ref{th:ModulesAndTableaux} is the only reason we need to work with skew tableaux, we will always assume that all skew tableaux are semi-standard and satisfy the lattice property. 

\subsection*{Notation} Let $T$ be a skew tableau with content $\nu = (\nu_1.,,,.\nu_t)$. Then we denote by $\gamma_j(T)$ the cardinality of $\{1 \leq i \leq t \ | \ \nu_i \geq j \}$. Note that this is essentially the same notation as for $kC$-modules for $C_{p^n\ell} \cong C = \langle g \rangle$ and $p\nmid \ell$. This is justified by Theorem~\ref{th:ModulesAndTableaux} if we think of $T$ as the skew tableau associated to a direct summand of $M$ on which the $p'$-part of $g$ acts as a certain fixed $\ell$-th root of unity.

\subsection{Orders and lattices} Let $\OO$ be a complete discrete valuation ring with residue field $k$ and field of fractions $K$. Let $\pi$ denote a uniformizer for $\OO$. An \emph{$\OO$-order} is an $\OO$-algebra which is free and finitely generated as an $\OO$-module. If $\Lambda$ is an $\OO$-order, then a $\Lambda$-module $L$ which is free and finitely generated as an $\OO$-module is called a \emph{$\Lambda$-lattice}. When we say that  $\Lambda$ is \emph{an order in} a $K$-algebra $A$, and $L$ is \emph{a $\Lambda$-lattice in} an $A$-module $V$, we always mean that they are \emph{full} lattices, that is, $\Lambda$ spans $A$ and $L$ spans $V$ as a $K$-vector space. In the present paper, we will only consider orders in semisimple $K$-algebras $A$, and we will assume that $A$ is semisimple for the remainder of this section.
We use the notation ``$J(\Lambda)$'' for the Jacobson radical of $\Lambda$. Note that the simple $\Lambda$-modules, the simple $\Lambda/\pi\Lambda$-modules and the simple $\Lambda/J(\Lambda)$-modules are the same. They should not be confused with the simple $A$-modules. 

An $\OO$-order $\Lambda$ in $A$ is called \emph{hereditary} if any $\Lambda$-lattice is projective. The order $\Lambda$ is called \emph{maximal} if it is not contained in any larger order within $A$. Maximal orders are hereditary, and hereditary orders have a well-developed structure theory \cite[{Chapter \S{9}}]{Reiner}. In particular, hereditary orders contain all central primitive idempotents of $A$, and are therefore direct products of hereditary orders in simple $K$-algebras.
We will concentrate on the case where no matrix algebras over non-commutative skew fields occur in the Wedderburn-Artin decomposition of $A$. A hereditary order in a matrix algebra $M_n(E)$, where $E$ is a field extension of $K$ of finite degree, looks as follows (up to conjugation):
\[
   \Gamma = \left(  \begin{array}{ccccc} 
      \EE^{n_1\times n_1} &  (\pi') ^{n_1\times n_2} & \cdots & (\pi') ^{n_1\times n_{m-1}} & (\pi') ^{n_1\times n_m}\\
      \EE^{n_2\times n_1} &  \EE ^{n_2\times n_2} & \cdots & (\pi') ^{n_2\times n_{m-1}} & (\pi') ^{n_2\times n_m}\\
      \vdots & \vdots&\ddots & \cdots &\vdots \\
      \EE^{n_{m-1}\times n_1} &  \EE ^{n_{m-1}\times n_2} & \cdots &\EE ^{n_{m-1}\times n_{m-1}} & (\pi') ^{n_{m-1}\times n_m}\\
      \EE^{n_{m}\times n_1} &  \EE ^{n_{m}\times n_2} & \cdots &\EE ^{n_{m}\times n_{m-1}} & \EE ^{n_{m}\times n_m}\\
   \end{array} \right) \subseteq M_n(E)
\]
where $\EE$ is the maximal order in $E$, $\pi'$ is a uniformizer for $\EE$, and $m,n_1,\ldots,n_m\in \N$ are such that $n=n_1+\ldots+n_m$. In particular, $Z(\Gamma) = \EE$ is a maximal order. All indecomposable lattices over $\Gamma$ are irreducible, that is, they are isomorphic to lattices in $E^{1\times n}$. The lattices inside $E^{1\times n}$ form a chain under inclusion, and there are exactly $m$ non-isomorphic ones, corresponding to the $m$ distinct types of rows in the depicition of $\Gamma$ above. We will make extensive use of the paper \cite{RoggenkampDefectOne}, which describes blocks of group algebras of defect 1 in terms of a hereditary order $\Gamma$ containing the block.

In the proof of Theorem~\ref{thm:filtration} we will use the \emph{Heller operator}, denoted $\Omega$, which is defined as the kernel of a projective cover. In particular, we will consider $\Omega_{\Lambda}(L)$ for $\Lambda$-lattices $L$ and $\Omega_{\Lambda/\pi\Lambda}(M)$ for $\Lambda/\pi\Lambda$-modules $M$. Note that the subscripts ``$\Lambda$'' and ``$\Lambda/\pi\Lambda$'' matter here, since even though we can view any $\Lambda/\pi\Lambda$-module as a $\Lambda$-module, the projective $\Lambda$-modules (which are lattices) are not the same as the projective $\Lambda/\pi\Lambda$-modules, which means that the projective covers differ. We will use the well-known fact that there are epimorphisms \cite[Section 2.6]{Benson}
\[
   \Hom_\Lambda(\Omega_\Lambda M, N) \twoheadrightarrow \Ext^1_\Lambda(M,N) \textrm{ and } \Hom_{\Lambda/\pi\Lambda}(\Omega_{\Lambda/\pi \Lambda} M, N)  \twoheadrightarrow  \Ext^1_{\Lambda/\pi\Lambda}(M,N).
\]
We will make use of the explicit construction given in  \cite[Section 2.6]{Benson} that takes an element of $\Hom(\Omega M, N)$ and produces an extension of $M$ by $N$.

\subsection{Blocks of cyclic defect}
We recall some facts about the representation theory of blocks of cyclic defect relevant to us. All of this can be found in \cite[Chapter VII]{Feit82}. We will mostly be concerned with blocks of defect $1$, and a more character-theoretic account of this special case can be found in \cite[Chapter 11]{Navarro98}.

Let $\OO$ be a complete discrete valuation ring with field of fractions $K$ and residue field $k$ of characteristic $p$. We assume that both $K$ and $k$ are splitting fields for ~$G$. Let $B$ be a $p$-block of $\mathcal{O} G$ with cyclic defect group $D \cong C_{p^n}$. Denote by ${\rm Irr}(B)$ the irreducible ordinary characters in $B$ and by ${\rm IBr}(B)$ the irreducible $p$-Brauer characters in $B$. The \emph{Brauer tree} $(V,E)$ associated to $B$ is a tree in the sense of graph theory, i.e. a loop-free non-oriented graph, in which one vertex known as the \emph{exceptional vertex} $v_x$ carries an additional label $\ell \in \mathbb{N}$ known as the \emph{multiplicity} of $v_x$. Moreover, there are maps
\[\alpha: {\rm Irr}(B) \rightarrow V, \ \ \beta: {\rm IBr}(B) \rightarrow E \]
such that $\beta$ is bijective, while $\alpha$ is surjective and the only element of $V$ which possibly has more than one preimage under $\alpha$ is $v_x$, namely $|\alpha^{-1}(v_x)| = \ell$. The preimages of $v_x$ are known as \emph{exceptional characters}, the other elements of ${\rm Irr}(B)$ as \emph{non-exceptional}. If $\ell = 1$ there is no exceptional character. If $L$ is an $\mathcal{O}G$-lattice with character $\chi \in {\rm Irr}(B)$, then the composition factors of $k \otimes_\mathcal{O} L$ are exactly the simple $kG$-modules which correspond to the edges adjacent to $\alpha(\chi)$ via $\beta$, each with multiplicity $1$. The Brauer tree is a connected graph and we call vertices adjacent to only one edge \emph{leafs}.
The characters of non-exceptional vertices are $p$-rational, while if we let $R$ be the ring obtained from adjoining the values of an exceptional character to $\mathbb{Z}_p$ the ramification index of $p$ in $R$ equals $p^{n-1} e$ for some divisor $e$ of $p-1$. Set $m = \frac{p-1}{e}$. Then $m$ equals the number of edges in the Brauer tree and $\ell = p^{n-1}m$. If $B$ is the principal block of $\mathcal{O}G$, then $e = [N_G(D):C_G(D)]$. 

Let $S$ be the subset of ${\rm Irr}(B)$ consisting of non-exceptional characters, fix an exceptional character $\theta$ and identify $V$ with $S \cup \{ \theta \}$. If we associate to each vertex $\psi$ a sign $ \delta_\psi \in \{ \pm 1\}$ in such a way that vertices connected by an edge have different signs, then for any element $g \in G$ of order not divisible by $p$ we have
\begin{align}\label{eq:AltSumOnP'Els}
0 = \delta_\theta \theta(g) + \sum_{\psi \in S} \delta_\psi \psi(g).
\end{align}
Note that in \cite[Chapter 11]{Navarro98} the sign at the exceptional vertex is reversed compared to our notation.

\section{Representation theory of $C_p$}

In this section $\OO$ denotes a complete discrete valuation ring with field of fractions $K$ of characteristic zero, uniformizer $\pi$ and residue field $k=\OO/\pi \OO$ of characteristic $p>0$. For $\chi$ and $\psi$ ordinary characters of a group $G$ we write $(\chi,\psi)_G$ for the scalar product of these characters.

\begin{theorem}\label{thm:filtration}
   Assume that $\OO[\zeta_p]$ is a maximal order in $K(\zeta_p)$ and set $m=(K(\zeta_p) : K)$. Let $C$ be a cyclic group of order $p$.
   Given a $KC$-module $V$ with character $\chi$ and a $kC$-module $N$, there is an $\OO C$-lattice $L\subseteq V$ with $L/\pi L\cong N$ if and only if $N$ has a filtration
   \begin{equation}
       0=N_0\leq N_1 \leq \ldots \leq N_{e+1} =  N\quad\textrm{with $e=\frac{p-1}{m}$ }
   \end{equation}
   such that 
   \begin{equation}\label{eqn filtration}
       N_{i+1}/N_{i} \cong I_{m}^{(\eta_{i}, \chi)_C} \textrm{ for $1 \leq i \leq e$ and } N_{1}/N_0 \cong I_1^{(\mathbf{1}, \chi)_C}.
   \end{equation}
   Here $\mathbf{1}$ denotes the trivial character of $C$ and $\eta_1,\ldots,\eta_e\in \Irr_{\bar K}(C)\setminus \{1\}$ representatives for the $\Gal(\bar K /K)$-orbits of irreducible characters of $C$.
   %, and $(-,=)_C$ denotes the scalar product on class functions.
\end{theorem}
\begin{proof}
   Consider the ``only if''-direction first. Let $\eps_2,\ldots, \eps_{e+1}$ denote the central primitive idempotents of $KC$ indexed in such a way that $\eta_i(\eps_{i+1})\neq 0$ for $1\leq i \leq e$, and let $\eps_{1}$ denote the central primitive idempotent belonging to the trivial character. Suppose we are given a lattice $L\leq V$. Then we can  define $N_i$ to be the image in $N=L/\pi L$ of $L_i=L\cap(L(\eps_1+\ldots+\eps_{i}))$ for $1\leq i \leq e+1$. We set $N_0=0$. Then $N_{i}/N_{i-1}$ is the reduction modulo $\pi$ of $L_i/L_{i-1}$, which can be viewed as a lattice in $V\eps_i$. Since $\OO C \eps_i$ is either isomorphic to $\OO$ or to $\OO[\zeta_p]$, both of which are maximal orders by assumption, it follows that any lattice in $V\eps_i$ is free. Furthermore, a free $\OO C \eps_i$-module reduces to a direct sum of copies of $I_m$ if $i >1$ and $I_1$ if $i=1$ (since a free $\OO C\eps_i$-module of rank 1 reduces to a uniserial $kC$-module, and comparing dimensions tells us that this uniserial module must be either $I_m$ or $I_1$). Comparing dimensions then yields the multiplicities given in~\eqref{eqn filtration}.
               
   We prove the ``if''-direction by induction. Let $\Lambda$ be a local $\OO$-order in a semisimple $K$-algebra $A$, and let $\eps$ be a central primitive idempotent in $Z(A)$ such that $\Lambda \eps$ is a maximal order in $A \eps$. Then any $\Lambda\eps$-lattice is free. Assume moreover that $\Lambda/\pi \Lambda$ is isomorphic to $k[x]/(x^\ell)$ for some $\ell \geq 1$. Note that $\Lambda=\OO C$ satisfies these assumptions, as does $\Lambda=\OO C e$ for any idempotent $e\in Z(KC)$. Hence, when trying to find a lift for $N_i$ in $V(\eps_1+\ldots +\eps_i)$ we can set $\Lambda=\OO C (\eps_1+\ldots+\eps_i)$ and $\eps=\eps_i$ and assume $\Lambda/\pi\Lambda \cong k[x]/(x^\ell)$. We will try to find such a lift as an extension of a lift of $N_{i-1}$ in $V(\eps_1+\ldots +\eps_{i-1})$ and a lift of $N_i/N_{i-1}$ in $V\eps_i$. Below we will outline the procedure that accomplishes this.
   
   We first want to show that the Heller translate $\Omega_\Lambda(\Lambda\eps)$ is a free $\Lambda(1-\eps)$-lattice of rank one. By definition of $\Omega_\Lambda$ there is a short exact sequence
   \begin{equation}
       0\longrightarrow \Omega_\Lambda(\Lambda\eps) \longrightarrow \Lambda \longrightarrow \Lambda \eps \longrightarrow 0.
   \end{equation}
   Since $\Lambda$ is local, so is $\Lambda\eps$. So the $\Lambda$-module $\Lambda\eps/\pi\Lambda \eps$ must have simple top, and is therefore isomorphic to $k[x]/(x^j)$ for some $j\leq \ell$ (which is determined by the dimension of $A\eps$). After reducing the exact sequence modulo $\pi$ it follows that $\Omega_\Lambda(\Lambda\eps)/ \pi \Omega_\Lambda(\Lambda\eps)\cong k[x]/(x^{\ell-j})$. In particular, $\Omega_\Lambda(\Lambda\eps)$ has simple top and is therefore generated by a single element as a $\Lambda$-lattice.

   Tensoring the short exact sequence above by $K$ tells us that $K\otimes_\OO \Omega_\Lambda(\Lambda\eps) \cong A(1-\eps)$, since the other two terms become $A$ and $A\eps$, respectively. To summarize, we have shown that $\Omega_\Lambda(\Lambda\eps)$ is a $\Lambda$-lattice in $A(1-\eps)$ generated by a single element. Hence $\Omega_\Lambda(\Lambda\eps) \cong \Lambda (1-\eps)$. In particular, $\Omega_\Lambda(\Lambda\eps)$ can be viewed as a projective $\Lambda(1-\eps)$-lattice.

   Now assume that $N$ is a $\Lambda/\pi\Lambda$-module which has a submodule $N'$ such that $N'$ is the reduction of a $\Lambda$-lattice $L'$ in $V(1-\eps)$ and such that $N/N'$ is the reduction of a lattice $L''$ in $V\eps$. We can think of $N$ as an extension of $N/N'$ by $N'$, which corresponds to an element of $\Ext^1_{\Lambda/\pi\Lambda} (N/N', N')$. To get a lattice $L$ with the required properties, we will show that the natural map 
   \begin{equation}\label{eqn ext surjection}
        \Ext^1_\Lambda(L'', L') \longrightarrow \Ext^1_{\Lambda/\pi\Lambda} (N/N', N')
   \end{equation} 
   is surjective. To prove this, first note that there is a surjection $\Hom_\Lambda(\Omega_\Lambda L'', L') \twoheadrightarrow \Ext^1_\Lambda(L'', L')$, and $\Omega_\Lambda L''$ is a projective $\Lambda(1-\eps)$-lattice. It follows that 
   \begin{equation}
       \Hom_\Lambda(\Omega_\Lambda L'', L') = \Hom_{\Lambda(1-\eps)}(\Omega_\Lambda L'', L') \twoheadrightarrow  \Hom_{\Lambda/\pi\Lambda}(\Omega_\Lambda L''/\pi \Omega_\Lambda L'', L'/\pi L'), 
   \end{equation}
   and $\Omega_\Lambda L''/\pi \Omega_\Lambda L'' = \Omega_{\Lambda/\pi\Lambda} (L''/\pi L'')$. The non-trivial assertion here is the surjectivity of the map above, which follows from the projectivity of $\Omega_\Lambda L''$ over $\Lambda (1-\eps)$. So every element of $ \Hom_{\Lambda/\pi\Lambda}( \Omega_{\Lambda/\pi\Lambda} (L''/\pi L''), L'/\pi L')$ is the reduction modulo $\pi$ of an element of $\Hom_\Lambda(\Omega_\Lambda L'', L')$. In principle, after passing from $\Hom(\Omega (-),=)$ to $\Ext^1(-,=)$ as before we should get surjectivity of~\eqref{eqn ext surjection} above. However, we choose a more direct approach. Namely, we will explicitly construct a lift of an extension given by $\bar \alpha \in \Hom_{\Lambda/\pi\Lambda}( \Omega_{\Lambda/\pi\Lambda} (L''/\pi L''), L'/\pi L')$. By what we saw before, we can lift $\bar \alpha$ to a homomorphism $\alpha$ of lattices, giving rise to an $\alpha \in  \Hom_\Lambda(\Omega_\Lambda L'', L')$. Following \cite[Section 2.6]{Benson} the corresponding extensions are constructed as follows. We choose a projective $\Lambda$-lattice $Q$ with a map $\beta: Q \longrightarrow L'$ such that $\alpha\oplus \beta: \Omega_\Lambda L''\oplus Q \longrightarrow L'$ is surjective, so $L'\cong (\Omega_\Lambda(L'') \oplus Q) / \Ker (\alpha\oplus \beta)$. Then the extension given by $\alpha$ is 
   \begin{equation}
      0\longrightarrow (\Omega_\Lambda(L'') \oplus Q) / \Ker (\alpha\oplus \beta) \longrightarrow (P \oplus Q) / \Ker (\alpha\oplus \beta)  \stackrel{p\oplus 0}{\longrightarrow} L''\longrightarrow 0,
   \end{equation}
   where $p: \ P \longrightarrow L''$ is a projective cover (whose kernel is $ \Omega_\Lambda(L'') $, by definition). The extension corresponding to $\bar \alpha$, constructed in the same way, has middle term $(P/\pi P \oplus Q/\pi Q) / \Ker (\bar \alpha\oplus \bar \beta) $, which is clearly a quotient of $(P \oplus Q) / \Ker (\alpha\oplus \beta) $. Comparing dimensions yields that it is indeed the reduction modulo $\pi$ of $(P \oplus Q) / \Ker (\alpha\oplus \beta) $. That is, we have indeed constructed a lift of the extension corresponding to $\bar\alpha$.
\end{proof}

\begin{remark}\label{rem:MaxOrders}
   The condition that $\OO[\zeta_p]$ needs to be a maximal order will always be satisfied in the situations we are interested in. Specifically, if $\OO'$ is an unramified extension of $\Z_p$ and $\chi$ is an absolutely irreducible character in a block of defect 1, then $\OO=\OO'[\chi]$ is a maximal order, and so is $\OO[\zeta_p]$. To see this first notice that a non-maximal order will remain non-maximal after extensions of scalars. So we may pick $\OO'$ to be the completion of the maximal unramified extension of $\Z_p$, and $K'$ its field of fractions. If $B$ is a block of defect 1 defined over $\OO'$, then the Wedderburn decomposition of $K'B$ does not involve any matrix algebras over non-commutative skew-fields (since a skew-field over a finite extension of $\Q_p$ splits over $K'$ due to \cite[Corollary (31.10)]{Reiner}). By \cite[Theorem II]{RoggenkampDefectOne} the order $\varepsilon B$ is hereditary where $\varepsilon$ is the central primitive idempotent in $Z(K'B)$ corresponding to $\chi$. Hence $Z(\varepsilon B)=\OO'[\chi]$ is a maximal order. We also know that $\OO'[\chi] \subseteq \OO'[\zeta_p]$, since $\chi$ vanishes on all elements whose $p$-part is not conjugate to an element of the defect group (and all $p'$-roots of unity are already contained in $\OO'$). Hence $\OO[\zeta_p]=\OO'[\zeta_p]$ is a maximal order too. 
\end{remark}

In our application to units in principal blocks of defect 1 we will only use the following corollary of the theorem above where Remark~\ref{rem:MaxOrders} shows that the theorem is indeed directly applicable in this situation. 

\begin{corollary}\label{cor:FlorianTamelyRamfied}
   Let $p$ be an odd prime and $u \in \V(\mathbb{Z}G)$ a unit of order $pr$ such that $p$ does not divide $r$. Let $\mathcal{O}'$ a complete discrete valuation ring with residue field $k$ of characteristic $p>0$ containing a primitive $r$-th root of unity $\xi$ and such that $p$ is unramified in $\mathcal{O}'$. Let $K'$ denote the field of fractions of $\OO'$. Let $\mathcal{O}$ be a complete discrete valuation ring such that $\mathcal{O}' \subseteq \mathcal{O} \subseteq \mathcal{O}'[\zeta]$ for a primitive $p$-th root of unity $\zeta$ and set $e = (K:K')$. Denote by $\zeta_1,...,\zeta_e$ representatives of the orbits of ${\rm Gal}(K(\zeta)/K)$ on $\langle \zeta \rangle \setminus \{1\}$.
   
    Let $L$ be an $\mathcal{O}\langle u \rangle$-lattice with character $\theta$ and $M'$ the $k\langle u \rangle$-module $L/\pi L$, where $\pi$ is a uniformizer for $\OO$. Denote by $M$ the maximal direct summand of $M'$ on which the $p'$-part of $u$ acts as $\xi^j$ for some fixed integer $j$. Then $M$ has a filtration
   \[0 = M_0 \leq M_1 \leq \dots \leq M_{e+1} = M\]
   with $M_{1}/M_0 \cong I_1^{\mu(\xi^j,u,\theta)}$ and $M_{i+1}/M_{i} \cong I_m^{\mu(\xi^j\cdot \zeta_i, u, \theta)}$ for $1 \leq i \leq e$.
\end{corollary}

In the unramified situation the $\mathcal{O}C_p$-lattices are well-known, see e.g. \cite[Theorem 2.6]{HellerReiner}, but the fact we will need is a direct consequence of Theorem~\ref{thm:filtration}.

\begin{corollary}\label{cor:UnramifiedCase}
Let $p$ be an odd prime, $u \in \V(\mathbb{Z}G)$ a unit of order $pr$ for $p \nmid r$ and $\theta$ a $p$-rational complex irreducible character of $G$. 
Then there exists a complete discrete valuation ring $\mathcal{O}$ with residue field $k$ of characteristic $p$, with $p$ unramified in $\mathcal{O}$, which contains a primitive $r$-th root of unity $\xi$, so that there is an $\mathcal{O}G$-lattice $L$ with character $\theta$. Moreover, denoting by $M$ the $kG$-module obtained from $L$ and by $\zeta$ a primitive $p$-th root of unity for every integer $j$ we get 
\[\gamma_{2,j}(M) = \gamma_{p-1,j}(M) = \mu(\xi^j\zeta,u,\theta)\]
\end{corollary}
\begin{proof}
By a theorem of Fong, cf. \cite[Remark 6]{BachleMargolis4primaryII}, we can take $\mathcal{O} = \mathbb{Z}_p[\eta]$ for $\eta$ a primitive $s$-th root of unity and $s$ the maximal divisor of $|G|$ not divisible by $p$. In the situation of Corollary~\ref{cor:FlorianTamelyRamfied} we then have $e=1$ and $m=p-1$, so that the claim follows.
\end{proof}

\subsection{Reversing the lattice method}
We will use Corollary~\ref{cor:FlorianTamelyRamfied} to disprove the existence of non-trivial units of order $p$ in principal blocks of defect 1, using a method sometimes termed the ``lattice method''.
However, Theorem~\ref{thm:filtration} is an ``if and only if''-statement, and it is in fact possible to decide (one way or the other) whether a unit with given eigenvalues on the irreducible modules exists. That is, we can use this method to construct units in blocks of defect 1 defined over $\Z_p$, which could potentially lead to further counterexamples to the Zassenhaus conjecture, or perhaps a negative answer to the question whether the orders of torsion units in $\V(\mathbb{Z}G)$ are the same as the orders of elements of $G$. A particular case of this problem is the so-called Prime Graph Question which asks this for products of two primes. As yet, we have not found any examples where there are no other obstructions to the existence of a corresponding global unit.

\begin{remark}\label{rem:Roggenkamp}
   Let $B$ be a block of defect $1$ defined over $\OO$, and let $A$ denote its $K$-span. We will use the following facts:
   \begin{enumerate}
       \item There is a hereditary order $\Gamma\subset A$ containing $B$ such that $J(\Gamma)=J(B)$. See \cite[Theorem II(iv)]{RoggenkampDefectOne}.
       \item $\Gamma$ is the hereditary hull of $B$, and, in particular, $B\eps=\Gamma\eps$ for all central-primitive idempotents $\eps \in Z(A)$.
       This follows from \cite[Theorem II(iii)]{RoggenkampDefectOne}, if we take into account that a hereditary order is actually determined by its radical (so, using the notation of \cite{RoggenkampDefectOne}, $J(\Gamma) = J(B)$ implies $J(\Gamma\eps)=J(\Gamma)\eps= J(B)\eps=J(B\eps)$, which implies $\Gamma\eps=B\eps$). 
       \item Let $S$ be a simple $B$-module and $P$ its projective cover. Then $P$ is a pullback of
       $$
           L_1 \twoheadrightarrow S \twoheadleftarrow L_2,
       $$
       where $L_1$ and $L_2$ are irreducible $B $-lattices with non-isomorphic $K$-spans. In particular, $L_1$ and $L_2$ are projective $\Gamma$-lattices, and $\Gamma\cdot P \cong L_1\oplus L_2$ (since $\Gamma\cdot P$ clearly has epimorphisms onto both $L_1$ and $L_2$). This follows from \cite[(v) and (vi)]{RoggenkampDefectOne}.

       \item For any simple $\Lambda$-module $S$ there are exactly two isomorphism classes of simple $\Gamma$-modules whose restriction to $\Lambda$ is isomorphic to $S$. This is a consequence of the previous point.
   \end{enumerate}
\end{remark}

\begin{lemma}\label{lemma:conjIntoHereditary}
   Let $A=M_n(K)$  and let $\Gamma\subset A$ be a hereditary order. Let $V$ denote the (unique up to isomorphism) simple $A$-module, and let $L_0=\pi L_r\subset L_1\subset \ldots \subset L_r$ be a chain of $\Gamma$-lattices in $V$ such that each $L_i$ is maximal in $L_{i+1}$. 
   
   Assume that $U$ is a finite group and $\varphi: U \longrightarrow \U(A)$ is a group homomorphism turning $V$ into a $KU$-module. Assume moreover that there is an $\OO U$-lattice $M\subset V$ such that $M/\pi M$
   has a filtration $0=N_0 \subset N_1 \subset \ldots \subset N_r=M/\pi M$, where $\dim_k(N_{i+1}/N_i) = \dim_k(L_{i+1}/L_i)$ for each $i$. Then there exists a unit $a\in A$ such that $a\varphi(U) a^{-1}\subseteq \Gamma$ and $L_ia=N_i$ for all $0\leq i\leq r$.
\end{lemma}
\begin{proof}
   We can find an $a\in \U(A)=\GL_n(K)$ such that $L_r=Ma$. After replacing $\varphi$ by a conjugate, we can therefore assume that $L_r=M=\OO^n$. In particular, we get two filtrations of $M/\pi M$, namely the one given by the $L_i+\pi M/\pi M$ and the one given by the $N_i$.
   Since $\GL_n(k)$ acts transitively on flags, and $\GL_n(\OO)$ surjects onto $\GL_n(k)$, we can find a $y\in \GL_n(\OO)\subseteq \U(A)$ such that $L_iy+\pi M/\pi M=N_i$ for all $0\leq i \leq r$. Again, after conjugating $\varphi$ appropriately we can assume that 
   $L_i+\pi M/\pi M=N_i$ for all $0\leq i \leq r$, and therefore each $L_i$ is stable under the action of $\varphi(U)$. Note that the $\Gamma$-lattices in $K^n$ form a chain, and therefore the $L_i$ are actually representatives for all isomorphism classes of irreducible $
   \Gamma$-lattices. It follows that 
   \begin{equation}
       \varphi(U) \subseteq \bigcap_{i=1}^r \End_\OO(L_i) =\Gamma,
   \end{equation}
   since $\Gamma$ is hereditary. This show that  $\varphi$ (which was obtained from the original $\varphi$ by conjugation) takes values in $\Gamma$.
\end{proof}

\begin{lemma}\label{lemma:fromhereditarytoblock}
   Let $B$ be a block of defect 1, and assume that the semisimple algebra $A=KB$ does not involve any matrix algebras over non-commutative skew-fields. Let $\Gamma\subseteq A$ be a hereditary $\OO$-order containing $B$ such that $J(B)=J(\Gamma)$. 
   Also assume that $U$ is a finite group and  $\varphi: U \longrightarrow \U(\Gamma)$ is a group homomorphism. We will view any $\Gamma$-module $M$ as an $\OO U$-module which we will denote $\Res_{\OO U} M$.
   
   Then there is an $a\in \U(A)$ such that $a\varphi(U) a^{-1}\subseteq B$ if and only if $\Res_{\OO U} S \cong \Res_{\OO U} T$ for any two simple $\Gamma$-modules $S$ and $T$ with $\Res_B S \cong \Res_B T$.
\end{lemma}
\begin{proof}
   Note that $\Gamma$ exist by Remark~\ref{rem:Roggenkamp}. The equality $J(B)=J(\Gamma)$ implies that we have an embedding of $k$-algebras $B/J(B)\hookrightarrow \Gamma/J(\Gamma)$. 
   An element $x\in \Gamma$ lies in $B$ if and only if $B x = B$, which happens if and only if $B x/J(B) = B/J(B)$. Now $B/J(B)$ is semisimple, and a simple direct summand $S$ of $B /J(B)$ satisfies $S\cdot \Gamma = T_1\oplus T_2$ for two non-isomorphic simple $\Gamma$-modules $T_1$ and $T_2$ by Remark~\ref{rem:Roggenkamp}. Note that $\Res_B T_1 \cong \Res_B T_2\cong S$.

   We also know that $\eps B = \eps \Gamma$ for every central-primitive idempotent $\eps \in Z(A)$. Being a hereditary order, $\Gamma$ is actually the direct product of the various $\eps \Gamma$, which means in particular that each simple $\Gamma$-module is actually a simple $\eps\Gamma$-module for some $\eps$. Since $B$ maps surjectively onto each $\eps\Gamma$ it follows that simple $\Gamma$-modules remain simple upon restriction to $B$. In particular $\End_{B}(\Res_\Lambda T) = \End_\Gamma(T)$ for any simple $\Gamma$-module $T$ (``$\supseteq$'' is clear, and if the containment was proper, the double centraliser theorem would imply that the image of $B$ in $\End_k(T)$ was strictly contained in that of $\Gamma$).
   
   Let us now consider the embedding of $k$-algebras $B/J(B)\hookrightarrow \Gamma/J(\Gamma)$. Both of these are semisimple $k$-algebras. We have
   \begin{equation}\label{eqn sdkjdoipjf}
       B /J(B) \cong \prod_{S} M_{n_S}(k_S) \hookrightarrow  \Gamma/J(\Gamma) \cong \prod_{T} M_{n_T}(k_T),
   \end{equation}
   where $S$ and $T$ run over representatives of the isomorphism classes of simple $B$- and $\Gamma$-modules, respectively, $k_S$ and $k_T$ denote $\End_B(S)^{\opp}$ and $\End_\Gamma(T)^{\opp}$, respectively, and $n_S$ and $n_T$ are the dimensions of these simple modules over $k_S$ and $k_T$. We will identify $B/J(B)$ and $\Gamma/J(\Gamma)$ with the corresponding direct product of matrix rings (this involves a choice of isomorphisms).
   
   The considerations above imply that whenever $\Res_B (T)\cong S$ for a simple $B$-module $S$ and simple $\Gamma$-module $T$, then $k_S=k_T$ and $n_S=n_T$. Furthermore, there are exactly two simple $\Gamma$-modules (up to isomorphism) whose restriction to $B$ is isomorphic to $S$. The central-primitive idempotent in $B/J(B)$ belonging to $S$ is therefore mapped to the sum of the central-primitive idempotents belonging to these two simple $\Gamma$-modules. It follows that the map from above restricts to a map
   \begin{equation}\label{eqn dkjoidjod}
       M_{n_S}(k_S) \hookrightarrow M_{n_S}(k_S) \times M_{n_S}(k_S).
   \end{equation}
   Composing this map with the two projections, and using the fact that the only algebra endomorphisms of $M_{n_S}(k_S)$ are inner automorphisms, it follows that this map is given by $M \mapsto (AMA^{-1}, \tilde{A}M\tilde{A}^{-1})$ for certain matrices $A,\tilde{A}\in \GL_{n_S}(k_S)$. We get one such matrix for each component of $\Gamma/J(\Gamma)$. Hence we get an element of $\bar a \in \U(\Gamma/J(\Gamma))$, such that, after replacing $B$ by $aB a^{-1}$ (where $a$ is an arbitrary preimage of $\bar a$ in $\U(\Gamma)$), the embedding \eqref{eqn dkjoidjod} is given by $M \mapsto (M,M)$. So, in order to check that an element $x\in \Gamma$ lies in $B$, all we have to check is that its image in $\Gamma/J(\Gamma)$ has identical entries in the components labelled by simples whose restriction becomes isomorphic upon restriction to $B$. 

   Let us consider $\bar \varphi:\ U \longrightarrow \Gamma /J(\Gamma): \ u \mapsto \varphi(u) + J(\Gamma)$. This is the representation attached to  $\Res_{\OO U} \left( \bigoplus_{T} T\right)$, where $T$ runs over all isomorphism classes of simple $\Gamma$-modules. For a simple $\Gamma$-module $T$ define $\bar \varphi_T$ to be the composition of $\bar \varphi$ with the projection $\Gamma/J(\Gamma)\twoheadrightarrow M_{n_T}(k_T)$. By the assumption, whenever the restriction of two non-isomorphisc simple $\Gamma$-modules $T$ and $\tilde{T}$ to $B$ is isomorphic, then so is their restriction to $\OO U$. In particular, we can find elements $A_T \in \GL_{n_T}(k_T)$ and $A_{\tilde{T}} \in \GL_{n_{\tilde{T}}}(k_{\tilde{T}})$ such that $A_T\bar \varphi_T(u)A_T^{-1}= A_{\tilde{T}}\bar \varphi_{\tilde{T}}(u)A_{\tilde{T}}^{-1}$ for all $u\in U$. We now pick a preimage of $\bar a=(A_T)_T \in \prod_{T} \GL_{n_T}(k_T)$, which we identified with $\U(\Gamma/J(\Gamma))$, in $\U(\Gamma)$. Denote this preimage by $a$. If we replace $\varphi$ by $a\varphi a^{-1}$, we get that $\bar \varphi_T$ and $\bar \varphi_{\tilde{T}}$ are equal whenever $T$ and $\tilde{T}$ have isomorphic restrictions to $B$. By the discussions in the previous paragraph it follows that, after replacing $\varphi$ by this conjugate, the image of $\bar\varphi$ is contained in $B/J(B)$,

   To finish, since $\bar \varphi$ has image contained in $B/J(B)$, the image of $\varphi$ is contained in $B + J(\Gamma)=B+J(B)=B$, which is what we needed to show.
\end{proof}

The two lemmas above allow us to decide when an embedding $U \longrightarrow \Q_p Gb$ is conjugate to an embedding $U \longrightarrow \Z_pGb$, assuming that $\Z_pGb$ is a block of defect 1, and $\Q_pGb$ does not involve any matrix algebras over non-commutative skew-fields. If we take $U=\langle u\rangle $ to be a finite cyclic group, then this provides a converse to the ``lattice method''.  Note that in the theorem below, $\Irr(B)$ and $\IBr(B)$ denote absolutely irreducible ordinary and modular characters. Over a splitting system, $B$ may split up into Galois conjugate blocks, and the Brauer ``tree'' referenced in the third point of the following theorem is actually the union of the Brauer trees of these split blocks. 

\begin{theorem}\label{thm:reversaloflatticemethod}
   Let $B$ be a block of defect 1 defined over $\OO$ such that the Wedderburn-Artin decomposition of $K B$ does not involve any matrix algebras over non-commutative skew-fields.  Let $\xi_1,\ldots,\xi_{r}$ denote the $r$-th roots of unity in $\bar K$, where $r\in \N$ is coprime to $p$. Let $\bar \xi_1,\ldots,\bar\xi_r$ denote the images of these roots of unity in $\bar k$.
   Since $\Gal(\bar K/K)$ and $\Gal(\bar k/k)$ act on these roots of unity, we get corresponding actions on the indices \{1,\ldots,r\}. 
   Let $\zeta$ be a primitive $p$-th root of unity in $\bar K$. 

   Let $u\in K B$ be a unit of order $pr$, and let $C$ denote a cyclic group of order $p$. Then $u$ is conjugate to a unit in $B$ if and only if, for each $1\leq j \leq r$, we can attach a $\bar kC$-module $S_{\psi, j}$ to each Brauer character $\psi \in \IBr(B)$ and a $\bar kC$-module $M_{\chi, j}$ to each ordinary character $\chi\in \Irr(B)$ such that the following hold:
   \begin{enumerate}
      \item For each $\chi\in \Irr(B)$ and $1\leq j\leq r$ we have 
      \[
         M_{\chi,j} \cong M_{\sigma(\chi),\sigma(j)} \quad \textrm{ for all $\sigma\in \Gal(\bar K/K)$.}
      \]
      \item For each $\psi\in\IBr(B)$  and $1\leq j\leq r$ we have 
      \[S_{\psi,j} \cong S_{\sigma(\psi),\sigma(j)} \quad \textrm{ for all $\sigma\in\Gal(\bar k/k)$.}\]
       \item For each $\chi\in \Irr(B)$ and $1\leq j \leq r$ there is a filtration
       \[0=N_0\leq N_1\leq \ldots \leq N_{a(\chi)}=M_{\chi,j},\] where $a(\chi)$ equals the number of edges in the Brauer tree adjacent to $\chi$, and $N_i/N_{i-1}\cong S_{\psi_i,j}$ for $i=1,\ldots, a(\chi)$, where  $\psi_1,\ldots, \psi_{a(\chi)}\in \IBr(B)$ label the edges adjacent to $\chi$ in counter-clockwise order. 
       \item For each $\chi\in \Irr(B)$ and $1\leq j \leq r$ set $m=(K(\chi,\zeta) : K(\chi))$, $e=\frac{p-1}{m}$ and let $\zeta_1,\ldots,\zeta_e$ be a system of representatives for the $\Gal(\bar K/K(\chi))$-orbits of primitive $p$-th roots of unity in $\bar K$. Then there is a filtration
       \[0=L_0\leq L_1\leq \ldots \leq L_{e+1}=M_{\chi,j},\] where 
       $L_{i+1}/L_{i} \cong I_{m}^{\mu(\xi_j\zeta_i, u, \chi)}$   for all 
       $1\leq i\leq e$ and $L_1/L_0 \cong I_{1}^{\mu(\xi_j, u, \chi)}$. 
   \end{enumerate}
\end{theorem}
\begin{proof}
   Let us first prove the ``only if'' direction. 
   If $u$ is conjugate to a unit in $B$, then we may replace $u$ by its conjugate in $B$ and assume without loss of generality that $u\in B$. 
   We can write $u=u_pu_r$, where $u_p,u_r\in B$ are units of order $p$ and order $r$, respectively. We identify $C=\langle u_p\rangle$.
   For each $\psi\in \IBr(B)$ and $1\leq j \leq r$ we can consider the corresponding simple $\bar k\otimes_\OO B$-module $S_{\psi}$,  and take the direct summand of $S_{\psi}$ on which $u_r$ acts as $\bar \xi_j$. This is a $\bar k\langle u_p\rangle = \bar kC$-module, and we may take this module as $S_{\psi,j}$. 

   To define the $M_{\chi,j}$ let us consider an unramified extension $\mathcal E$ of $\mathcal O$ such that $\mathcal E/\pi\mathcal E \cong \bar k$. Let $E$ denote the field of fractions of $\mathcal E$. Denote by $\pi'$ a uniformizer for $\mathcal E[\chi]$. Note that $\mathcal E[\chi]/\pi'\mathcal E[\chi] \cong \mathcal E /\pi\mathcal E \cong \bar k$. There is a simple $E(\chi)\otimes_\OO B$-module $V_\chi$ whose character is $\chi$. Pick an $\mathcal E[\chi]\otimes_{\OO} B$-lattice $L_\chi$ in $V_\chi$ and define $L_{\chi,j}$ as the direct summand of 
   $L_\chi$ on which $u_r$ acts like $\xi_j$. Then set $M_{\chi,j} = L_{\chi,j}/\pi' L_{\chi,j}$. We may view $L_{\chi,j}$ as an $\mathcal E[\chi]\langle u_{p}\rangle=\mathcal E[\chi] C$-lattice.
   Using the fact that orbits of $\Gal(\bar K /E(\chi))$ on $p$-th roots of unity are the same  as the orbits of $\Gal(\bar K/K(\chi))$ we get the required filtration for the fourth point directly from Theorem~\ref{thm:filtration}.
   Moreover, $M_{\chi,j}$ has a filtration by the $S_{\psi,j}$ as required since $M_\chi=L_\chi/\pi' L_\chi$ viewed as a $\bar k\otimes_\OO B$-module has a corresponding filtration by the $S_\psi$ by the definition of the Brauer tree (which may possibly start at a different Brauer character than $\psi_1$, but we can use Lemma~\ref{lem:ModSwap} to rearrange the filtration of $M_{\chi,j}$ as needed). 
   
  The Galois invariance of the $S_{\psi,j}$ from the second point follows directly from the definition of the $S_{\psi,j}$. The Galois invariance of the $M_{\chi,j}$ from the first point is not immediate from the construction above (and may in fact fail), but we can simply define $M_{\chi,j}$ for representatives of the $\Gal(\bar K/K)$-orbits on $\Irr(B)\times\{1,\ldots,r\}$ using the construction above, and then define $M_{\sigma(\chi), \sigma(j)}=M_{\chi, j}$ whenever $(\chi,j)$ is a representative (a more natural construction would be to choose the lattice $L_{\sigma(\chi)}$ in $V_{\sigma(\chi)}$ corresponding to $L_\chi$ in $V_\chi$, but that requires additional set-up).

   For the ``if''-direction, assume that we are given $M_{\chi,j}$ and $S_{\psi,j}$ as in the statement. 
   For each $\chi$ consider $\varepsilon_\chi \OO B$, where $\varepsilon_\chi \in Z(KB)$ is the central primitive idempotent corresponding to the $\Gal(\bar K/K)$-conjugacy class of $\chi$. 
   The center $Z(\varepsilon_\chi \OO B)= \OO[\chi]$ is a maximal order by
   % \cite{RoggenkampDefectOne} 
   Remark~\ref{rem:MaxOrders} combined with our assumption on the Wedderburn decomposition, and we may therefore view $\varepsilon_\chi \OO B$ as an $\OO[\chi]$-order.
   Set $k'\cong \OO[\chi]/J(\OO[\chi])$ and write $\pi'$ for a generator of $J(\OO[\chi])$. It follows from  \cite[Theorem II(i)]{RoggenkampDefectOne} that $k'$ is independent of $\chi$, as it is isomorphic to the endomorphism ring of a simple $B$-module (all of which have isomorphic endomorphism rings).
   
   We will first consider the unit $\varepsilon_\chi u \in \varepsilon_\chi K B$. The latter is a full matrix algebra over $K(\chi)$ due to the assumption that no matrix algebras over skew-fields occur in the Wedderburn decomposition. We denote the unique simple $\varepsilon_\chi K B$-module by $V_\chi$.
   
   Now let us pick a $k'C$-module $M'_{\chi,j}$ which becomes isomorphic to $M_{\chi,j}$ upon extension of scalars from $k'$ to $\bar k$. The module $M'_{\chi,j}$ still has a filtration as in the fourth point, since the existence of such a filtration is determined combinatorially in terms of tableaux, and is therefore independent of the base field. Now we can use Theorem~\ref{thm:filtration} to get an $\OO[\chi]\langle u_p\rangle$-lattice $L_{\chi,j}$ with $L_{\chi,j}/\pi' L_{\chi,j} \cong M'_{\chi,j}$ such that the $\zeta_i$-eigenspace of $u_p$ acting on $\bar K \otimes_{\OO[\chi]} L_{\chi,j}$ has dimension $\mu(\xi_j\zeta_i, u, \chi)$ for all $1\leq i\leq e$. We turn that into the  $\OO[\chi]\langle u\rangle=\OO[\chi](\langle u_p\rangle\times \langle u_r\rangle)$-lattice $L_{\chi,j} \otimes_{\OO[\chi]} \OO[\chi,\xi_j]$, where we let $u_r$ act on $\OO[\chi,\xi_j]$ like $\xi_j$. Then set
   \[
      L_{\chi} = \bigoplus_{j} L_{\chi,j} \otimes_{\OO[\chi]} \OO[\chi,\xi_j] 
   \]
   where $j$ ranges over representatives for the $\Gal(\bar K/K(\chi))$-orbits on $\{1,\ldots,r\}$. By definition, the multiplicity of the eigenvalue $\xi_j\zeta_i$ on $\bar K\otimes_{\OO[\chi]}L_{\chi}$ is $\mu(\xi_j\zeta_i, u, \chi)$, which means that $K(\chi)\otimes_{\OO[\chi]} L_\chi \cong V_\chi$ (regarded as a $K(\chi)\langle u\rangle$-module).

   By the third point, $L_\chi/\pi' L_\chi$ has a filtration with subquotients isomorphic to 
   \[
      S_\psi = \bigoplus_{j} S_{\psi,j} \otimes_{k'} k'[\bar{\xi}_j]
   \]
   where $j$ ranges over representatives for the $\Gal(\bar k/k')$-orbits on $\{1,\ldots,r\}$, which are equal to the $\Gal(\bar K/K(\chi))$-orbits on $\{1,\ldots,r\}$, since the $p'$-roots of unity in $\bar K$ map bijectively onto those in $\bar k$, and this bijection maps the $p'$-roots of unity in $K(\chi)$ onto those in $k'$. This means that this definition of $S_\psi$ is in fact independent of $\chi$. It now follows that $\dim_{k'}S_{\psi}=\psi(1)$, since the existence of the above filtration implies the equations $\dim_{k'}S_{\psi_1}+\ldots+\dim_{k'}S_{\psi_{a(\chi)}}=\rank_{\OO[\chi]} L_\chi = \chi(1)$ for all $\chi\in\Irr(B)$ (using the notation of the third point), and $\dim_{k'}S_{\psi}=\psi(1)$ is the unique solution to that. Of course $\psi(1)$ is also the $k'$-dimension of a simple $B$-module corresponding to $\psi$.
   
   Let us again fix a $\chi\in \Irr(B)$. Using Lemma~\ref{lemma:conjIntoHereditary} with $U=\langle \eps_\chi u \rangle$ we now get a conjugate of $\varepsilon_\chi u$ inside of $\varepsilon_\chi B$ such that the restriction of the simple $\varepsilon_\chi B$-module corresponding to $\psi_i$ (using the notation from the third point) is isomorphic to $S_{\psi_i}$. Since we can do this for all $\chi$, we get a unit $u'$ conjugate to $u$ in 
   \[
      \Gamma = \bigoplus_{\chi} \varepsilon_\chi B,
   \]
   where $\chi$ ranges over representatives for the $\Gal(\bar K/K)$-orbits on $\Irr(B)$. Lemma~\ref{lemma:fromhereditarytoblock} shows that $u'$ is in fact conjugate to a unit in $B$, as required.
\end{proof}

A typical envisaged application of Theorem~\ref{thm:reversaloflatticemethod} is given in Example~\ref{ex:PSL216}. In particular, the reader might notice that the conditions (1) and (2) are necessary to ensure the existence of a unit conjugate to $u$ in $B$, but they are not necessary to check in practice for a unit which is known to exist in $\mathbb{Q}G$, since we simply define the $M_{\chi,j}$ and $S_{\psi,j}$ on representatives of Galois orbits.

% as they are automatically satisfied then for $M_{\chi,j}$ being the part of a simple $\mathbb{C}G$-module with character $\chi$ on which $u^p$ acts as a certain fixed $r$-th root of unity and $S_{\psi,j}$ is the part of a simple $\bar{\mathbb{F}}_pG$-module with Brauer character $\psi$ on which $u^p$ acts with the same $r$-th root of unity. Hence for applications one needs only find the filtrations from points (3) and (4) for representatives of the Galois orbits.

\section{Combinatorics}

In this section we will prove the combinatorial facts we will need for our proofs later. We start by recalling some results from \cite{CaicedoMargolis}. Even when not stated explicitly all skew tableaux appearing are assumed to be semi-standard and to satisfy the lattice property.

\begin{lemma}\label{lem:ModSwap}\cite[Lemma 2.8]{CaicedoMargolis}
Let $k$ be a field of characteristic $p$, $C$ a cyclic group of order $p$ and let $M$, $Q_1,\ldots ,Q_n$ be $kC$-modules. If $\sigma$ is a permutation of $\{1,\ldots,n \}$ then there exists a filtration
\[
   0=M_0\leq M_1\leq \ldots\leq M_n=M \quad \textrm{ with $M_i/M_{i-1} \cong Q_i$ for all $1\leq i \leq n$}
\]
if and only if there exists a filtration
\[
   0=N_0\leq N_1\leq \ldots\leq N_n=M \quad \textrm{ with $N_i/N_{i-1} \cong Q_{\sigma(i)}$ for all $1\leq i \leq n$}.
\]
\end{lemma}

\begin{lemma}\label{lem:full_rectangle}\cite[Lemma 3.4]{CaicedoMargolis}
 Let $T$ be a semistandard skew tableau satisfying the lattice property. If $T$ contains a full rectangle of boxes of height $h$ and width $n$, then $\gamma_n(T)\ge h$.
\end{lemma}

\begin{lemma}\label{lem:colums_between_lines}\cite[Lemma 3.5]{CaicedoMargolis}
 Let $T$ be a semistandard skew tableau with $\ell$ columns satisfying the lattice property. Fix some positive integers $c$, $h$ and $n$. Assume that the first $\ell-n$ columns of $T$ lie between the $(c+1)$-th and $(c+h)$-th rows. Then $\gamma_{n+1}(T)\le h$.
\end{lemma}

\[
\begin{tikzpicture}
\draw  (0,2) -- (0,0) -- (3,0) -- (3,1) -- (4,1) -- (4,2) -- (5,2) -- (5,4) -- (4,4) -- (4,3) -- (1,3) -- (1,2) -- (0,2
);

%horizontal lines
\draw[red, dashed] (-0.5,3.5) -- (5.5, 3.5);
\draw[red, dashed] (-0.5,0) -- (5.5, 0);
%vertical lines 
\draw[red, dashed] (0,-0.5) -- (0, 4.5);
\draw[red, dashed] (3.5,-0.5) -- (3.5, 4.5);

%labels 
\node[label=south:{\Small{$\ell - n$} }] at (3.5, 5) {};
\node[label=west:{\Small{$c+1$} }] at (-0.5,3.5) {};
\node[label=west:{\Small{$c+h$} }] at (-0.5,0) {};
\node[label=north:{$T$ }] at (2,1.5) {};
 
\end{tikzpicture}
\]

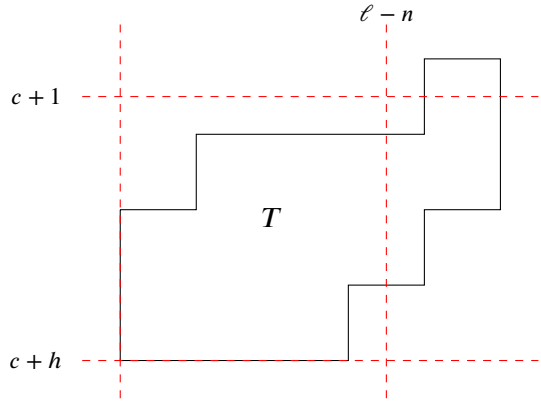
\captionof{figure}{Illustrating Lemma~\ref{lem:colums_between_lines}}

\begin{lemma}\label{lem:divided_tableau} \cite[Lemma 3.6]{CaicedoMargolis}
 Let $T$ be a semistandard skew tableau satisfying the lattice property. Divide $T$ into two skew tableaux $T'$ and $T''$ by a vertical line such that $T''$ consists of $\ell$ columns and $T'$ is the tableau on the right. Then $T'$ is again a semistandard skew tableau with the lattice property and $\gamma_{n+\ell} (T) \leq \gamma_n(T')$.
\end{lemma}

The following are new results of purely combinatorial nature which will apply to our situation.

\begin{lemma}\label{lem:RectangelAndLine}
Let $a,b,c$ be positive integers. Assume $T = T_1/T_2$ is a semistandard skew tableau satisfying the lattice property such that $T_2$ is a rectangle of height $b$ and width $c$, the first $b$ rows of $T_1$ contain at least $a + c$ boxes and the $(b+1)$-th row contains $a+c-1$ boxes, cf. Figure~\ref{fig:cutABox}. Then $\gamma_a(T) > b$ or $\gamma_{a+c}(T) > 0$.  
\end{lemma}

\begin{proof}
The entries in the first $b$ rows are clear: the $i$-th row contains the entry $i$ in each box. So in this part the entries $1,2,...,b$ come all up at least $a$ times. In the $(b+1)$-th row the $(a-1)$ boxes in the east contain $b+1$ as an entry. The westward neighbor of these boxes can contain either also $b+1$, in which case $\gamma_a(T) > b$, or it contains $1$. In the latter case also the boxes to the west contain $1$, so there are $c$ boxes in this row containing $1$ and at least $a$ boxes containing $1$ in the first row, implying $\gamma_{a+c}(T) > 0$.
\end{proof}

\[
\begin{tikzpicture}
\draw  (0,2) -- (0,0) -- (1,0) -- (1,1) -- (2.5,1) -- (2.5,2);
\draw (4.5, 2.5) -- (4,2.5) -- (4,2)  (0,2) -- (0,2.5) -- (3,2.5) -- (3,3.5);
\draw (2.5,2) -- (4,2)  (3,3.5) -- (4.5,3.5)  (4.5,3) -- (4.5,2.5);
\draw (4.5,3.5) -- (5,3.5) -- (5,3) -- (4.5,3);

%%horizontal lines
\draw[red, dashed] (-0.5,2.5) -- (5.5, 2.5);
\draw[red, dashed] (-0.5,2) -- (5.5, 2);
%%vertical lines 
\draw[red, dashed] (3,-0.5) -- (3, 4.5);
\draw[red, dashed] (4.5,-0.5) -- (4.5, 4.5);
\draw[red, dashed] (4,1) -- (4, 4.5);

%labels 
\node[label=north:{\Small{$c$ }}] at (3, 4.3) {};
\node[label=north:{\Small{$a+c$} }] at (4.8, 4.3) {};
\node[label=west:{\Small{$b$} }] at (-0.5,2.5) {};
\node[label=east:{\Small{$b+1$} }] at (5.5,2) {};
\node[label=south:{\Tiny{$a+c-1$} }] at (4,1.2) {};
\node[label=north:{$T$ }] at (1.5,1.2) {};
 
\end{tikzpicture}
\]

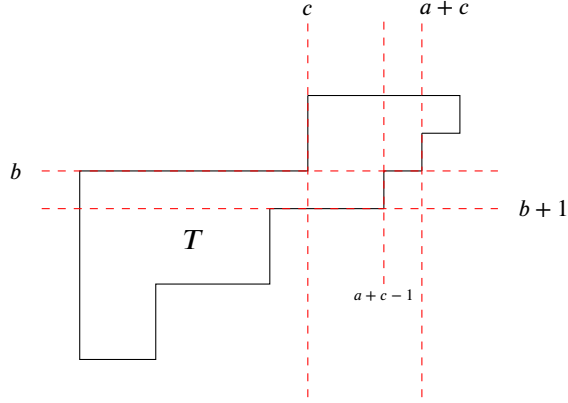
\captionof{figure}{Illustrating Lemma~\ref{lem:RectangelAndLine}}\label{fig:cutABox}

\begin{lemma}\label{lem:TooBigIndSums}
Let $m$ and $e$ be positive integers such that $m = \frac{p-1}{e}$. Assume $M$ is a $kC_p$-module with filtration 
\[0 \subseteq M_0 \subseteq ...\subseteq M_e = M\]
such that $M_0$ is the direct sum of trivial modules while $M_i/M_{i-1}$ is the direct sum of indecomposable modules of dimension $m$ for each $1 \leq i \leq e$. Then $\gamma_{2+m\cdot i}(M_i) = 0$ for each $0 \leq i \leq e$
\end{lemma}
\begin{proof}
We argue by induction. As $M_0$ is the direct sum of trivial modules the case $i = 0$ is clear. So assume that $\gamma_{2+m\cdot(i-1)}(M_{i-1}) = 0$ and $\gamma_{2+m\cdot i}(M_i) > 0$. Then the first row of the semistandard skew tableau satisfying the lattice property associated to of $M_i/M_{i-1}$ has length bigger than $2 + m \cdot i - (2 + m \cdot (i-1)) = m$. As all the entries of the first row of this skew tableau are $1$, this implies that $M_i/M_{i-1}$ has a direct indecomposable summand of dimension bigger than $m$, contradicting the properties of the filtration.
\end{proof}

% \begin{remark}\label{rem:TooBigIndSums}
%    In Lemma~\ref{lem:TooBigIndSums}, if we ask instead that for some $j>0$, the quotient $M_{j}/M_{j-1}$ is a direct sum of trivial modules and $M_0$ is a direct sum of indecomposable modules of dimension $m$, then the conclusion of the lemma remains valid for all $i\geq j$. This follows easily using Lemma~\ref{lem:ModSwap}.
% \end{remark}

\begin{lemma}\label{lem:CombOnM}
Let $m$ and $e$ be positive integers such that $m = \frac{p-1}{e}$. Assume $M$ is a $kC_p$-module with filtration 
\[0 \subseteq M_0 \subseteq ...\subseteq M_e = M\]
such that $M_0$ is the direct sum of trivial modules while $M_i/M_{i-1} \cong I_m^{\mu_i}$ for each $1 \leq i \leq e$ for some non-negative integers $\mu_1,...,\mu_e$. Then
\begin{itemize}
\item[(i)] $\gamma_{m}(M) \geq max\{\mu_i \ | \ 1 \leq i \leq e\}$,
\item[(ii)] $\gamma_{p-m+1}(M) \leq min\{\mu_i \ | \ 1 \leq i \leq e\}$.
\end{itemize}
\end{lemma}
\begin{proof}
To see (i) let $\mu = max\{\mu_i \ | \ 1 \leq i \leq e\}$. By Lemma~\ref{lem:ModSwap} we can permute the filtration of $M$ so that we know it contains a submodule isomorphic to $I_m^{\mu}$. It follows directly that $\gamma_m(M) \geq \mu$.

To prove (ii) we assume that $\mu_e = min\{\mu_i \ | \ 1 \leq i \leq e\}$, which again is possible by Lemma~\ref{lem:ModSwap} after permuting the filtration. Assume that $\gamma_{p-m+1}(M) > \mu_e$. By Lemma~\ref{lem:ModSwap} we can write $M_{e-1}$ as a quotient of the form $M_e/I_m^{\mu_e}$. If we consider a skew tableau of this quotient, we see that the $(\mu_e+1)$-th row has at least $p-m+1$ boxes. This row is a rectangle of height $1$ and width at least $p-m+1$, so that we conclude from Lemma~\ref{lem:full_rectangle} that $\gamma_{p-m+1}(M_{e-1}) > 0$. But as $p-m+1 = 2 + m \cdot (e-1)$ this contradicts Lemma~\ref{lem:TooBigIndSums}.
\end{proof}

Our final goal of this section is the generalization of \cite[Proposition 4.3]{CaicedoMargolis}, which is given as Proposition~\ref{prop:general_tree} below, to the more general situation around an exceptional vertex. The following notation is fixed for the rest of this section: $p$ is an odd prime, $u \in \V(\mathbb{Z}G)$ is a unit of order $pr$ for $r$ not divisible by $p$. We denote by $\xi$ a fixed $r$-th root of unity and by $\zeta$ a primitive $p$-th root of unity. $B$ denotes a block of cyclic defect and $\mathcal{O}$ a complete discrete valuation ring containing a primitive $|G|_{p'}$-th root of unity with residue field $k$ of characteristic $p$ such that any complex irreducible representation of $B$ can be realized over $\mathcal{O}$ and $p$ is as unramified as possible in $\mathcal{O}$. By a theorem of Fong this is equivalent to $\mathcal{O}$ containing the character values of the exceptional characters of $B$, cf. \cite[Remark 6]{BachleMargolis4primaryII}. We denote by $\mathcal{O}'$ the maximal subring of $\mathcal{O}$ in which $p$ is unramified. By the same theorem of Fong all the non-exceptional characters of $B$ can be realized over $\mathcal{O}'$.

\begin{proposition}\label{prop:general_tree}\cite[Proposition 4.3]{CaicedoMargolis}
Let $M$ be an irreducible $\mathcal{O}'G$-lattice with character $\chi$ lying in $B$. Assume that $D$ is a composition factor of $\bar{M}$ when viewed as $kG$-module. Let $S$ denote the subtree of the Brauer tree of $B$ consisting of the vertices lying to the same side of $D$ as $\chi$. That is, $S$ is the connected component containing $\chi$ of the graph obtained by removing $D$.
We also view $\chi$ as an element of $S$ and moreover assume that $S$ does not contain the exceptional vertex. Let $a$ be the number of vertices in $S$. Then the following hold:
\begin{itemize}
   \item[(i)] If $\delta_{\chi}=-1$, then $\gamma_{p-a, \xi}(D)\ge - \sum_{\psi\in S}\delta_{\psi}\cdot \mu(\xi \cdot \zeta, u, \psi)$,
   \item[(ii)] If $\delta_{\chi}=1$, then $\gamma_{a+1, \xi}(D)\le  \sum_{\psi\in S}\delta_{\psi}\cdot \mu(\xi \cdot \zeta, u, \psi)$.
  \end{itemize} 
  
  If there is a leaf $\chi_1$ of the Brauer tree of $B$ lying in $S$ with positive sign (i.e. $\delta_{\chi_1}=1$), then the following hold as well: 
\begin{itemize}
   \item[(iii)] If $\delta_{\chi}=-1$, then $\gamma_{p-a+1,\xi}(D)\ge - \mu(\xi , u, \chi_1) - \sum_{\psi\in S}\delta_{\psi}\cdot \mu(\xi \cdot \zeta, u, \psi)$,
   \item[(iv)] If $\delta_{\chi}=1$, then $\gamma_{a,\xi}(D)\le \mu(\xi , u, \chi_1) + \sum_{\psi\in S}\delta_{\psi}\cdot \mu(\xi \cdot \zeta, u, \psi)$.
\end{itemize}

%Assume $p$ is not ramified in $\mathcal{O}$. Let $M$ be a simple $\mathcal{O}G$-module with character $\chi$ lying in a $p$-block $B$ of cyclic defect. Assume that $D$ is a composition factor of $\bar{M}$ when viewed as $kG$-module. We denote the subtree of the Brauer tree of $B$ consisting of the vertices lying to the same side of $D$ as $\chi$ as $S$. I.e. $S$ is the connected component containing $\chi$ of the graph one would get when $D$ would be removed.
%We also view $\chi$ as an element of $S$ and moreover assume that $S$ does not contain the exceptional vertex. Let $a$ be the number of vertices in $S$. 

% \begin{itemize}
% \item[(1.a)]\label{a} If $\delta_{\chi}=-1$, then $\gamma_{p-a, \xi}(D)\ge - \sum_{\psi\in S}\delta_{\psi}\cdot \mu(\xi \cdot \zeta_p, u, \psi)$,
% \item[(1.b)] If $\delta_{\chi}=1$, then $\gamma_{a+1, \xi}(D)\le  \sum_{\psi\in S}\delta_{\psi}\cdot \mu(\xi \cdot \zeta_p, u, \psi)$.
% \end{itemize}
%
%Moreover, if $S$ contains a leaf, say $\chi_1$, of the whole Brauer tree of $B$ such that $\chi_1$ is labeled by a positive sign $+1$ , then: 
%\begin{itemize}
% \item[(2.a)] If $\delta_{\chi}=-1$, then $\gamma_{p-a+1,\xi}(D)\ge - \mu(\xi , u, \chi_1) - \sum_{\psi\in S}\delta_{\psi}\cdot \mu(\xi \cdot \zeta_p, u, \psi)$,
% \item[(2.b)] If $\delta_{\chi}=1$, then $\gamma_{a,\xi}(D)\le \mu(\xi , u, \chi_1) + \sum_{\psi\in S}\delta_{\psi}\cdot \mu(\xi \cdot \zeta_p, u, \psi)$.
%\end{itemize} 

\[
\begin{tikzpicture}
\draw (0,0) -- (1.4,0) -- (1.4,1.4) -- (0,1.4) -- (0,0);
\draw (1.4,0.7) -- (3.5, 0.7);

%%%%%%%

\node at (1.4,0.7) (1){};
\node at (3.5, 0.7) (2){};
\node at (0.7, 0.7) (3){};

%%%%%%%

\node[label=south:{$\chi$ }] at (1.8, 0.8) {};
\node[label=north:{$D$ }] at (2.5,0.5) {};
\node[label=north:{$S$ }] at (0.7,0.3) {};

%%%%%%%

\draw[fill=white]  (1) circle (.1cm);
\draw[fill=white]  (2) circle (.1cm);
 
\end{tikzpicture}
\] 

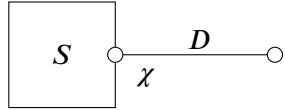
\captionof{figure}{Illustrating Proposition~\ref{prop:general_tree}}\label{BrauerBox}
\end{proposition}

\begin{proposition}\label{prop:MultSumIneq}
Let $\theta$ be an exceptional character of $B$ and denote by $M$ a $kG$-module which is the reduction of an irreducible $\mathcal{O}G$-lattice with character $\theta$. Let $m$ be the number of edges of the Brauer tree. Denote by $S$ the set of non-exceptional characters in $B$. Then the following hold:
\begin{itemize}
\item[(i)] If $\delta_\theta = -1$, then $\gamma_{m+1, \xi}(M) \leq \sum_{\psi \in S} \delta_\psi \mu(\xi\cdot \zeta, u, \psi)$.
\item[(ii)] If $\delta_\theta = 1$, then $\gamma_{p-m,\xi}(M)  \geq - \sum_{\psi \in S} \delta_\psi \mu(\xi\cdot \zeta, u, \psi)$.
\end{itemize}
If there is a leaf $\chi_1$ of the Brauer tree of $B$ with positive sign (i.e. $\delta_{\chi_1}=1$), then the following hold as well: 
\begin{itemize}
\item[(iii)] If $\delta_\theta = -1$, then $\gamma_{m,\xi}(M)  \leq \mu(\xi,u,\chi_1) + \sum_{\psi \in S} \delta_\psi \mu(\xi\cdot \zeta, u, \psi)$. 
\item[(iv)] If $\delta_\theta = 1$, then $\gamma_{p-m+1,\xi}(M)  \geq -\mu(\xi,u,\chi_1) - \sum_{\psi \in S} \delta_\psi \mu(\xi\cdot \zeta, u, \psi)$. 
\end{itemize}

\end{proposition}

\begin{proof}
The proof is similar to that of Proposition~\ref{prop:general_tree}, relying on the Lemmas \ref{lem:full_rectangle}, \ref{lem:colums_between_lines} and \ref{lem:divided_tableau}. As we will only consider the maximal direct summand of $M$ on which the $p'$-part of $u$ acts as $\xi$, we will denote this summand simply by $M$ and omit the subscript $\xi$ also in $\gamma$.

Let the edges adjacent to the exceptional vertex be labeled by $kG$-modules $D_1$,...,$D_n$ and let $B_1$,...,$B_n$ be the subtrees attached to these edges on the other side, cf. Figure~\ref{fig:BrauerBoxProof}. 

\[
\begin{tikzpicture}

%horizontal line 
\draw (0,1.5)--(2,1.5);
% oblique lines
\draw (0.5,0.25)--(2,1.5);
\draw (0.5,2.75)--(2,1.5);

%squares
% top
\draw (-0.3, 2.35) -- (0.5, 2.35) -- (0.5, 3.15) -- (-0.3, 3.15) -- (-0.3, 2.35);
%middle
\draw (-0.8, 1.1) -- (0,1.1) -- (0,1.9) -- (-0.8, 1.9) -- (-0.8, 1.1);
%bottom
\draw (-0.3, -0.15) -- (0.5,-0.15) -- (0.5, 0.65) -- (-0.3, 0.65) -- (-0.3, -0.15);

%nodes
 \node at (2,1.5) (2){};

%circles
 \foreach \p in {2}{
    \draw[fill=white] (\p) circle (.08cm);
 }

%labels
 \node[label=north:{\Small{$D_1$} }] at (1.3,2) {};
 \node[label=north:{\Small{$D_i$} }] at (0.9,1.3) {};
 \node[label=north:{\Small{$D_n$} }] at (1.4,0.3) {};
 \node[label=north:{\Small{$M$} }] at (2.1,1.4) {};
 \node[label=west:{\Small{$B_1$} }] at (0.6,2.7) {};
 \node[label=west:{\Small{$B_i$} }] at (0.05,1.5) {};
\node[label=west:{\Small{$B_n$} }] at (0.6,0.2) {};
 \end{tikzpicture}
\]

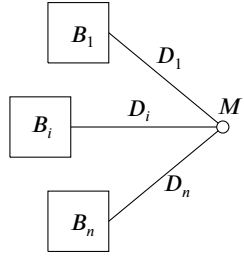
\captionof{figure}{Illustrating proof of Proposition~\ref{prop:MultSumIneq}}\label{fig:BrauerBoxProof}
\vspace*{.5cm}

We write $|B_i|$ for the number of vertices in $B_i$. We assume that $B_n$ contains $\chi_1$, if it exists. Note that 
\[m = |B_1| + ... + |B_n|.\]
Fix a filtration \[0=M_0\leq M_1\leq M_2\leq \ldots\leq M_n = M\] such that $M_i/M_{i-1}\cong D_i$ for $1\leq i \leq n$.
Let $\alpha \in \{0,1 \}$. 

Assume first $\delta_\theta = -1$, so the signs of the neighbors of $\theta$ are positive. Hence, one gets from Proposition~\ref{prop:general_tree} that $\gamma_{|B_i| +1}(D_i) \leq \sum_{\psi \in B_i} \delta_\psi \mu(\xi \cdot \zeta,u,\psi)$. We will show by induction that for any $0 \leq i \leq n-1$ one has
\begin{align}\label{eq:indneg}
\gamma_{m+\alpha - |B_1| - ... - |B_i|}(M/M_i) \geq \gamma_{m+\alpha}(M) - \sum_{\psi \in B_1 \cup ...\cup B_i} \delta_\psi \mu(\xi \cdot \zeta,u,\psi).
\end{align}
The base case $i=0$ is trivial, since the statement becomes $\gamma_{m+\alpha}(M)\geq \gamma_{m+\alpha}(M)$. 
For the induction step let $T$ be a skew tableau associated to the quotient of  $M/M_{i-1}$ by $M_i/M_{i-1}$. By the induction hypothesis, the diagram associated to $M/M_{i-1}$ contains in the northwest a rectangle $R$ of height $\gamma_{m+\alpha}(M)- \sum_{\psi\in B_1 \cup \ldots \cup B_{i-1}} \delta_\psi \mu(\xi \cdot \zeta,u,\psi)$ and width $m+\alpha-|B_1| - \ldots - |B_{i-1}|$. As $\gamma_{|B_i|+1}(D_i) \leq \sum_{\psi\in B_i} \delta_\psi \mu(\xi \cdot \zeta,u,\psi)$ by Proposition~\ref{prop:general_tree} the part of $R$ which survives in $T$ contains in its southeast corner a full rectangle of height $\gamma_{m+\alpha}(M)- \sum_{\psi\in B_1 \cup \ldots \cup B_{i}} \delta_\psi \mu(\xi \cdot \zeta,u,\psi)$ and width $m+\alpha-|B_1|-\ldots-|B_i|$, cf. Figure~\ref{fig:IndStartNeg}. So by Lemma~\ref{lem:full_rectangle} we conclude 
\[\gamma_{m+\alpha-|B_1|-\ldots-|B_i|}(M/M_i) \geq \gamma_{m+\alpha}(M) - \sum_{\psi\in B_1 \cup \ldots \cup B_i} \delta_\psi \mu(\xi \cdot \zeta,u,\psi),\]
 which finishes the induction step.

 \[
\begin{tikzpicture}
\draw  (0,2) -- (0,-1) -- (2,-1) -- (2,0) -- (3,0) -- (3,1) -- (4,1) -- (4,2) -- (5,2) -- (5,4) -- (4,4) -- (4,3) -- (1,3) -- (1,2) -- (0,2);

%horizontal lines
\draw[dashed] (-0.5,2.5) -- (3.5, 2.5);
\draw[dashed] (-0.5,1.5) -- (3.5, 1.5);
%vertical lines 
\draw[dashed] (1.5,3) -- (1.5, 4.5);
\draw[dashed] (3.5,3) -- (3.5, 5);
\draw[red,dashed] (1.5,2.5) -- (1.5, 3);
\draw[red,dashed] (3.5,2.5) -- (3.5, 3);
%labels
\node[label=north:{\Small{$m+\alpha-|B_1|-\ldots-|B_{i-1}|$} }] at (3.5, 5) {};
\node[label=north:{\Small{$|B_i|$} }] at (1.5, 4.5) {};
\node[label=west:{\Small{$\sum_{\psi\in B_i} \delta_\psi \mu(\xi \cdot \zeta,u,\psi)$} }] at (-0.5,2.5) {};
\node[label=west:{\Small{$\gamma_{m+\alpha}(M)- \sum_{\psi\in B_1 \cup \ldots \cup B_{i-1}} \delta_\psi \mu(\xi \cdot \zeta,u,\psi)$} }] at (-0.5,1.5) {};
\node[label=south:{\Small{$T$} }] at (1.5, 1) {};
%rectangle
\draw[red] (1.5,1.5)  -- (3.5, 1.5) -- (3.5, 2.5) -- (1.5,2.5) -- (1.5,1.5);
\end{tikzpicture}
\]

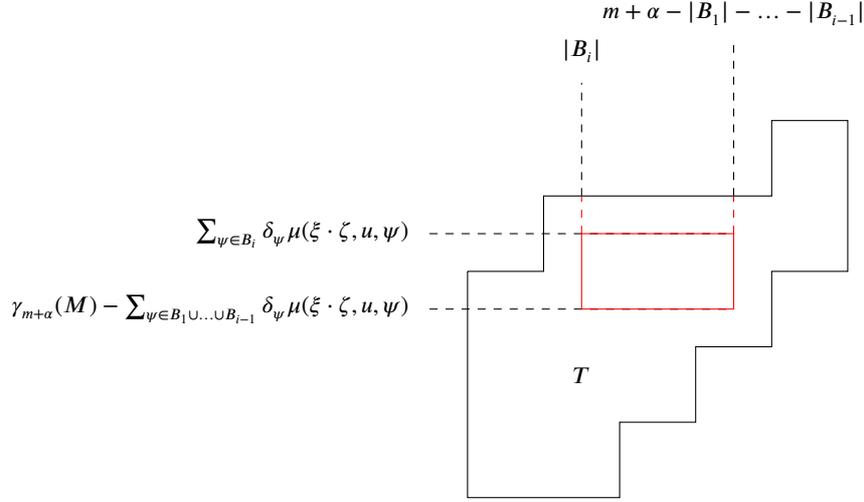
\captionof{figure}{Illustrating the case $\delta_\theta = -1$ in the proof of Proposition~\ref{prop:MultSumIneq}}\label{fig:IndStartNeg}
\vspace*{.5cm}

As $m + \alpha - |B_1| - ... - |B_{n-1}| = |B_n| + \alpha$ and $M/M_{n-1} \cong D_n$ we obtain from the inequality \eqref{eq:indneg} in case $i = n-1$ that
\[\gamma_{|B_n| + \alpha} (D_n) \geq \gamma_{m+\alpha}(M) - \sum_{\psi \in B_1 \cup ...\cup B_{n-1}} \delta_\psi \mu(\xi\cdot\zeta,u,\psi). \]
By Proposition~\ref{prop:general_tree} we know $\gamma_{|B_n| + 1}(D_n) \leq \sum_{\psi \in B_n} \mu(\xi\cdot \zeta,u,\psi)$, so by setting $\alpha=1$ this inequality implies (i). On the other hand, also by Proposition~\ref{prop:general_tree}, we know  $\gamma_{|B_n|}(D_n) \leq \mu(\xi,u,\chi_1) + \sum_{\psi \in B_n} \mu(\xi\cdot\zeta,u,\psi)$, so that setting $\alpha=0$ implies (iii). 

%Relying on this one shows by induction using Lemma~\ref{lem:full_rectangle} that for any $1 \leq i \leq n-1$
%\[\gamma_{m+\alpha - |B_1| - ... - |B_i|}(M/D_1/.../D_i) \geq t - \sum_{\psi \in B_1 \cup ...\cup B_i} \delta_\psi \mu(\zeta,u,\psi)  \]
%This is done by finding in the $i$-th step in the skew tableau a full rectangle of height $t - \sum_{\psi \in B_1 \cup ...\cup B_i} \delta_\psi \mu(\zeta,u,\psi)$ and width $m + \alpha - |B_1| - ...- |B_i|$.
%Once the induction is established, one uses $m + \alpha - |B_1| - ... - |B_{n-1}| = |B_n| + \alpha$ and $M/D_1/.../D_{n-1} \cong D_n$ to get
%\[\gamma_{|B_n| + \alpha} (D_n) \geq t - \sum_{\psi \in B_1 \cup ...\cup B_{n-1}} \delta_\psi \mu(\zeta,u,\psi). \]
%Then one uses the inequalities from Proposition~\ref{prop:general_tree} that $\gamma_{|B_n| + 1}(D_n) \leq \sum_{\psi \in B_n} \mu(\zeta,u,\psi)$ and $\gamma_{|B_n|}(D_n) \leq 1 + \sum_{\psi \in B_n} \mu(\zeta,u,\psi)$ to prove (i) and (ii).

Now assume $\delta_\theta = 1$, so the signs of the neighbors of $\theta$ are negative. By Proposition~\ref{prop:general_tree} we have $\gamma_{p-|B_i|}(D_i) \geq -\sum_{\psi \in B_i}\delta_\psi \mu(\xi\cdot \zeta,u,\psi)$. We will show by induction that
\begin{align}\label{eq:indpos}
\gamma_{p-m+\alpha+|B_1| + ... + |B_i|}(M/M_i) \leq \gamma_{p-m+\alpha}(M) + \sum_{\psi \in B_1 \cup ... \cup B_i}\delta_\psi \mu(\xi\cdot\zeta,u,\psi)
\end{align}
for $0 \leq i \leq n-1$ under the assumptions of (ii) and (iv), respectively. Again, the base case $i=0$ is trivial.

For the induction step consider a skew tableau $T$ corresponding to the quotient of  $M/M_{i-1}$ by $M_i/M_{i-1}$. Divide $T$ into two skew tableaux $T'$ and $T''$ such that $T''$ lies on the left and contains $p-m-1+\alpha+|B_1|+...+|B_{i-1}|$ columns. As $\gamma_{p-|B_i|}(D_i)  \geq  -\sum_{\psi \in B_i}\delta_\psi \mu(\xi\cdot \zeta,u,\psi)$ by Proposition~\ref{prop:general_tree} the first
\[p-|B_i| - (p-m-1+\alpha+|B_1|+...+|B_{i-1}|) = m+1-\alpha-|B_1|-...-|B_i| \]
columns of $T'$ contain no box in the first $ -\sum_{\psi \in B_i}\delta_\psi \mu(\xi\cdot \zeta,u,\psi)$ rows. By the induction hypothesis 
\[\gamma_{p-m+\alpha+|B_1|+...|B_{i-1}|}(M/M_{i-1}) \leq \gamma_{p-m+\alpha}(M) +  \sum_{\psi \in B_1 \cup ...\cup B_{i-1}}\delta_\psi \mu(\xi\cdot \zeta,u,\psi),\]
 which means that the first $m+1-\alpha-|B_1|-...-|B_i|$ columns of $T'$ lie between the $(1-\sum_{\psi \in B_i}\delta_\psi \mu(\xi\cdot \zeta,u,\psi))$-th and $\gamma_{p-m+\alpha}(M) + \sum_{\psi \in B_1 \cup ...\cup B_{i-1}}\delta_\psi \mu(\xi\cdot \zeta,u,\psi)$ rows, cf. Figure~\ref{fig:IndStartPos}. Moreover, $T'$ has at most
\[p-(p-m-1+\alpha+|B_1|+...+|B_{i-1}|) = m + 1-\alpha - |B_1|-...-|B_{i-1}|\]
columns. Hence by Lemma~\ref{lem:colums_between_lines} we know $\gamma_{|B_i|+1}(T') \leq\gamma_{p-m+\alpha}(M) + \sum_{\psi \in B_1 \cup ...\cup B_i}\delta_\psi \mu(\xi\cdot \zeta,u,\psi)$ and by Lemma~\ref{lem:divided_tableau} the induction claim follows.

\[
\begin{tikzpicture}
\draw  (0,2) -- (0,-1) -- (1,-1) -- (1,0) -- (3,0) -- (3,1) -- (4,1) -- (4,2) -- (5,2) -- (5,4) -- (4,4) -- (4,3) -- (1,3) -- (1,2) -- (0,2);

%horizontal lines
\draw[dashed] (-0.5,3.5) -- (1.5, 3.5);
\draw[red, dashed] (1.5, 3.5)-- (3.5, 3.5);
\draw[dashed] (-0.5,-0.5) -- (1.5,-0.5); 
\draw[red, dashed] (1.5,-0.5) -- (3.5, -0.5);

%vertical lines 
\draw[dashed]  (1.5,3.5) -- (1.5, 4.5);
\draw[line width=0.1cm, gray] (1.5,4) -- (1.5, -1);
\draw[dashed] (3.5,5) -- (3.5, 3.5);
\draw[red,dashed] (3.5,3.5)--(3.5,-0.5);

%labels
\node[label=north:{\Small{$p-|B_i|$} }] at (3.5, 5) {};
\node[label=south:{\Small{$p-m-1+\alpha+|B_1|+...+|B_{i-1}|$} }] at (0.5, 5) {};
\node[label=west:{\Small{$-\sum_{\psi \in B_i}\delta_\psi \mu(\xi\cdot \zeta,u,\psi)$} }] at (-0.5,3.5) {};
\node[label=west:{\Small{$\gamma_{p-m+\alpha}(M) + \sum_{\psi \in B_1 \cup ... \cup B_{i-1}}\delta_\psi \mu(\xi\cdot\zeta,u,\psi)$} }] at (-0.5,-0.5) {};
\node[label=south:{\Small{$T''$} }] at (0.8, 1.3) {};
\node[label=south:{\Small{$T'$} }] at (2.5, 2) {};

\end{tikzpicture}
\]

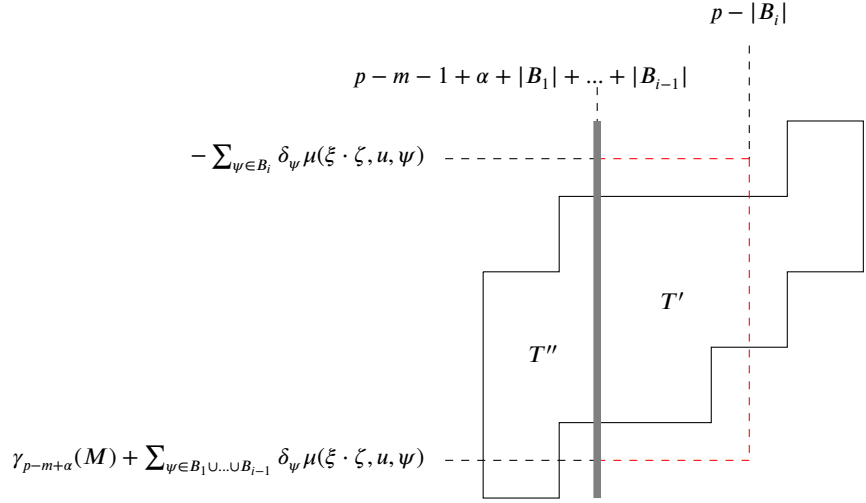
\captionof{figure}{Illustrating the case $\delta_\theta = 1$ in the proof of Proposition~\ref{prop:MultSumIneq}}\label{fig:IndStartPos}
\vspace*{.5cm}

Substituting $|B_1| + ... + |B_{n-1}| = m - |B_n|$ and $M/M_{n-1} \cong D_n$ in \eqref{eq:indpos} for $i = n-1$ we obtain
\[\gamma_{p-|B_n|+\alpha}(D_n) \leq \gamma_{p-m+\alpha}(M)   + \sum_{\psi \in B_1 \cup ... \cup B_{n-1}}\delta_\psi \mu(\xi\cdot \zeta,u,\psi). \]
Moreover, Proposition~\ref{prop:general_tree} gives $\gamma_{p-|B_n|}(D_n) \geq - \sum_{\psi \in B_n}\delta_\psi\mu(\xi\cdot \zeta,u,\psi)$ and $\gamma_{p-|B_n|+1}(D_n) \geq -\mu(\xi,u,\chi_1) - \sum_{\psi \in B_n}\delta_\psi\mu(\xi\cdot \zeta,u,\psi)$. Together with the previous inequality these give (ii) and (iv).
\end{proof}

\section{Sylow subgroups of order $p$}

In this section we will prove Theorem~\ref{thm:main}. We assume $G$ has a cyclic Sylow subgroup of order $p$ for some odd prime $p$. A unit $u \in \V(\mathbb{Z}G)$ of order $p$ is fixed throughout this section. The goal will be to show that  $u$ is rationally conjugate to a trivial unit. It turns out that it suffices to consider the principal block, as rational conjugacy within the principal block will constrain the partial augmentations, which in turn gives rational conjugacy within the whole group algebra. Note that for $p=2$ this is known and almost explicitly mentioned in \cite[Corollary 3.5]{Hertweck2006}. Also for $p=3$ this follows similarly: either there is only one conjugacy class of elements of order $3$ and so each $u \in \V(\mathbb{Z}G)$ of order $3$ is conjugate in $\mathbb{Q}G$ to a trivial unit or there are two conjugacy classes of elements of order $3$. In the latter case the Sylow $3$-subgroup has a normal $3$-complement $N$ by Burnside's $p$-complement theorem. But then a $u \in \V(\mathbb{Z}G)$ of order $3$ which is not rationally conjugate to a trivial unit, would map to a non-trivial unit in $\mathbb{Z}(G/N) \cong \mathbb{Z}C_3$ by \cite[Lemma 1.2]{DokuchaevJuriaans}. That such units do not exist was already proven by Higman \cite[Proposition 1.5.1]{GRG1}.

Let $e$ be the number of conjugacy classes of order $p$ in $G$ and $m = \frac{p-1}{e}$. If $e=1$ the result follows from Theorems~ \ref{th:MRSW} and \ref{th:pAsofTorsionUnits}, so we will assume $e \geq 2$. Let $\zeta$ be a primitive complex $p$-th root of unity and $\sigma \in Gal(\mathbb{Q}(\zeta)/\mathbb{Q})$ of order $m$. Let $\zeta_1,\zeta_2,...,\zeta_e$ be representatives of the orbits of $\langle \sigma \rangle$ on $\langle \zeta \rangle \setminus\{1\}$. Set
\[\tilde{\zeta}_i = \sum_{\tau \in \langle \sigma \rangle} \zeta_i^\tau.\]

We fix our notation for the Brauer tree of the principal block $B$ of $G$. We fix an exceptional character $\theta$ and denote by $S$ the set of non-exceptional characters. Note that $|S| = m$. We identify $S \cup \{\theta \}$ with the vertices of the Brauer tree. We also assume that the trivial character is labeled by a positive sign.
%We fix signs on the Brauer tree as done in \cite[Chapter VII, Theorem 2.15]{Feit82}\ChLeo{das sollte sicherlich imn den Grundlagen erklaert werden} and write $\delta_v$ for the sign at $v$. So these signs are alternating on the tree and for a $p'$-element $g \in G$ we have
%\begin{align}\label{eq:AltSumOnP'Els}
%0 = \delta_\theta \theta(g) + \sum_{\psi \in S} \delta_\psi \psi(g)
%\end{align}

%We denote by $I_s$ an indecomposable $kC_p$-module of dimension $s$. 

All skew tableaux appearing in this section are associated to $k\langle u \rangle$-modules and assumed to be semistandard and to satisfy the lattice property.

\subsection{Character values}

The following is a consequences of \cite[Theorem 11.1]{Navarro98}.

\begin{proposition}\label{prop:CharValuesPrinBlock}
There are representatives $g_1,...,g_e$ of all the conjugacy classes of order $p$ in $G$ such that
\[\theta(g_i) = - \delta_\theta \cdot \tilde{\zeta}_i. \]
Moreover for $\psi \in S$ we have $\psi(g_i) = \delta_\psi$ for any $i$.
\end{proposition}

%\begin{proposition}
%For an element $g \in G$ of order $p$ and $h \in C_G(g)$ for any $\chi \in Irr(G)$ in $B$ we have $\chi(gh) = \chi(g)$. in any character ring of a character in the principal block $p$ is tamely ramified or unramified. 
%\end{proposition}

Let from now on $g_1,\ldots,g_e$ be representatives of conjugacy classes of elements of order $p$ in $G$ as in Proposition~\ref{prop:CharValuesPrinBlock}. From Theorem~\ref{th:pAsofTorsionUnits} we know that $\varepsilon_g(u) = 0$ for $g$ not conjugate to any $g_i$. 

Let $D$ be a representation of $G$ realizing $\theta$ over a complete discrete valuation ring $\mathcal{O}$ containing a primitive $|G|_{p'}$-th root of unity such that $p$ is as unramified in $\mathcal{O}$ as possible, i.e. $[\mathcal{O}[\zeta]:\mathcal{O}] = m$. Moreover, $k$ denotes the residue field of $\mathcal{O}$ and $M$ the $kG$-module one obtains by reducing an irreducible $\mathcal{O}G$-lattice corresponding to $D$ modulo the maximal ideal of $\mathcal{O}$. Note that as $u$ has order $p$ any eigenvalue of $D(u)$ is a $p$-th root of unity. Define the two functions
\[\min(u) = \min\{\mu(\zeta_r, u, \theta) \ | \ 1 \leq r \leq e \} \ \ \text{and} \ \ \max(u) = \max\{\mu(\zeta_r, u, \theta) \ | \ 1 \leq r \leq e \}  \]

Note that for any two indices $i$ and $j$,  $g_i$ is conjugate to a power of $g_j$, and therefore $\min(g_i) = \min(g_j)$ and $\max(g_i) = \max(g_j)$. We will denote these quantities by $\min(g)$ and $\max(g)$, respectively.
 Moreover Proposition~\ref{prop:CharValuesPrinBlock} implies $\max(g) - \min(g) = 1$, which also is given more explicitly in the next lemma.

 Note that in this setting Proposition~\ref{pr:luthar-passi-multiplicity-formula} implies that 
 \begin{equation}\label{eq:luthar-passi-special-case}
 \mu(\zeta_i, v, \chi) = \frac{1}{p}(\chi(1) + {\rm Tr}_{\mathbb Q(\zeta)/\mathbb Q} (\chi(v)\zeta_i^{-1}))
\end{equation}
for any ordinary character $\chi$ and any unit $v\in \V(\mathbb Z G)$ of order $p$. This, combined with the fact that $ {\rm Tr}_{\mathbb Q(\zeta)/\mathbb Q} (\zeta_i)=-1$ for all $i$ and ${\rm Tr}_{\mathbb Q(\zeta)/\mathbb Q}(1)=p-1$, will be used repeatedly to compute multiplicities. 

\begin{lemma}\label{lem:CharValOfU}
Let the situation be as above and $g \in G$ of order $p$. Then for any $i$ we have:
\begin{itemize}
\item[(i)] if $\delta_\theta = -1$, then $\min(g) = \frac{\theta(1)-m}{p} $ and $\mu(\zeta_i, u, \theta) =  \varepsilon_{g_i}(u) + \min(g)$. 
\item[(ii)] if $\delta_\theta = 1$, then $\max(g)=\frac{\theta(1) + m}{p} $ and $\mu(\zeta_i, u, \theta) = -\varepsilon_{g_i}(u) + \max(g)$. 
\end{itemize}
\end{lemma}

\begin{proof}
Using Proposition~\ref{prop:CharValuesPrinBlock} we get
\[ 
     \mu(\zeta_i,g_j,\theta)=\frac{1}{p}(\theta(1) -\delta_\theta\cdot {\rm Tr}_{\mathbb Q(\zeta)/\mathbb Q} (\tilde \zeta_j \zeta_i^{-1})) = \frac{1}{p}(\theta(1) -\delta_\theta\cdot (-m+p\cdot \delta_{i,j})) 
\]
where $\delta_{i,j}$ denotes the Kronecker delta. If $\delta_\theta=-1$ this attains its minimum when $i\neq j$, and if $\delta_\theta=1$ it attains its maximum when $i\neq j$. The formulas for $\min(g)$ and $\max(g)$ follow. 

Now use the fact that 
\[
   \theta(u)=-\delta_\theta\sum_{j=1}^e \varepsilon_{g_j}(u) \tilde \zeta_j \quad \textrm{and}\quad \sum_{j=1}^e \varepsilon_{g_j}(u)=1
\]
to compute
\[ 
   \begin{array}{rcl}
     \mu(\zeta_i,u,\theta)&=&\displaystyle \frac{1}{p}(\theta(1) -\delta_\theta\cdot \sum_{j=1}^e\varepsilon_{g_j}(u) {\rm Tr}_{\mathbb Q(\zeta)/\mathbb Q} (\tilde \zeta_j \zeta_i^{-1})) \\&=&\displaystyle \frac{1}{p}(\theta(1) -\delta_\theta\cdot \sum_{j=1}^e\varepsilon_{g_j}(u) (-m+p\cdot \delta_{i,j})) =
     \displaystyle \frac{1}{p}(\theta(1) -\delta_\theta\cdot (-m +\varepsilon_{g_i}(u)\cdot p)). 
   \end{array}
\]
Using the formulas $\min(g)$ and $\max(g)$ from before, the claimed formulas for $\mu(\zeta_i,u,\theta)$ follow.
\end{proof}

\begin{lemma}\label{lem:SumPsi1}
We have
\[\sum_{\psi \in S} \delta_\psi \mu(\zeta, u, \psi) = \frac{-\delta_\theta \theta(1) -m}{p} \]
and therefore
\begin{itemize}
\item[(i)] if $\delta_\theta = -1$, then $\sum_{\psi \in S} \delta_\psi \mu(\zeta, u, \psi) = \min(g)$.
\item[(ii)] if $\delta_\theta = 1$, then $-\sum_{\psi \in S} \delta_\psi \mu(\zeta, u, \psi) = \max(g)$.
\end{itemize}
\end{lemma}
\begin{proof}
As any $\psi \in S$ is $p$-rational, the value of $\psi$ on any conjugacy class of order $p$ is the same and so $\psi(u) = \psi(g_1)=\ldots=\psi(g_e)$ by Theorem~\ref{th:pAsofTorsionUnits}. As $\psi(g_i) = \delta_\psi$ (for any $i$), formula~\eqref{eq:luthar-passi-special-case} gives us $\mu(\zeta, u, \psi) = \frac{\psi(1) - \delta_\psi}{p}$. Hence, using equation~\eqref{eq:AltSumOnP'Els} and $|S| = m$ we get
\[\sum_{\psi \in S} \delta_\psi \mu(\zeta, u, \psi) = \sum_{\psi \in S} \delta_\psi \frac{\psi(1) - \delta_\psi}{p} = \frac{\sum_\psi \delta_\psi \psi(1) - 1}{p} = \frac{-\delta_\theta \theta(1) -m}{p}.  \]
The other two equations in the lemma then follow directly from Lemma~\ref{lem:CharValOfU}.
\end{proof}

\subsection{Estimating summands}

\begin{lemma}\label{lem:MinMax}
One has $\varepsilon_{g_i}(u) \leq 1$ for any $1 \leq i \leq m$. Moreover,
\begin{itemize}
\item[(i)] if $\delta_\theta = -1$, then $\gamma_m(M) = \max(u) = 1 + \min(g)$.
\item[(ii)] if $\delta_\theta = 1$, then $\gamma_{p-m+1}(M) = \min(u) = -1 + \max(g)$.
\end{itemize}
\end{lemma}
\begin{proof}
First assume $\delta_\theta = -1$. By Lemma~\ref{lem:CombOnM} we know $\gamma_m(M) \geq \max(u)$.
By Proposition~\ref{prop:MultSumIneq} and Lemma~\ref{lem:SumPsi1}  we have 
$\gamma_m(M) \leq 1 + \sum_{\psi \in S} \delta_\psi \mu(\zeta,u,\psi)\leq 1 +\min(g)$. By Lemma~\ref{lem:CharValOfU} we know $\mu(\zeta_i, u, \theta) = \varepsilon_{g_i}(u) + \min(g)$. As certainly $\varepsilon_{g_i}(u) \geq 1$ for some $i$, this implies $1 +\min(g) \leq \max(u)$. We conclude 
that 
\[\gamma_m(M) \leq 1 +\min(g) \leq \max(u) \leq \gamma_m(M),\]
which shows that all of these quantities are equal. If there was some $i$ such that $\varepsilon_{g_i}(u) > 1$, then Lemma~\ref{lem:CharValOfU}  would imply $\max(u)> 1+\min(g)$, contradicting the equalities we just proved.

When $\delta_\theta = 1$, by Lemma~\ref{lem:CombOnM} we have $\gamma_{p-m+1}(M) \leq \min(u)$. By Proposition~\ref{prop:MultSumIneq} and Lemma~\ref{lem:SumPsi1} we have  $\gamma_{p-m+1}(M) \geq - 1 - \sum_{\psi \in S} \delta_\psi \mu(\zeta,u,\psi)\geq -1+\max(g)$. By Lemma~\ref{lem:CharValOfU} we have $\mu(\zeta_i,u,\theta) = -\varepsilon_{g_i}(u) + \max(g)$ for any $i$. As some $\varepsilon_{g_i}(u)$ is positive, we conclude $-1 + \max(g)\geq \min(u)$, 
and therefore 
\[\gamma_{p-m+1}(M)\geq -1+\max(g)\geq \min(u)\geq \gamma_{p-m+1}(M),\]
proving that all of these are equal. Again we get $\varepsilon_{g_i}(u) \leq 1$ for all $i$ using Lemma~\ref{lem:CharValOfU}.
\end{proof}

\begin{lemma}\label{lem:difference2Case}
With the notation above assume that $u$ is not rationally conjugate to a trivial unit. 
\begin{itemize}
\item[(i)] If $\delta_\theta = -1$, then $\gamma_{m+1}(M) \geq \max(u)$.
\item[(ii)] If $\delta_\theta = 1$, then $\gamma_{p-m}(M) \leq \min(u)$.
\end{itemize}
\end{lemma}

\begin{proof}
By Lemma~\ref{lem:MinMax} we know $\varepsilon_{g_i}(u) \leq 1$ for any $i$. Hence if $u$ is not rationally conjugate to a trivial unit, i.e. there is a negative partial augmentation by Theorem~\ref{th:MRSW}, there are at least two partial augmentations of value $1$, since the total augmentation of $u$ equals $1$. Note that in particular $e \geq 2$.

First assume that $\delta_\theta = -1$. By Lemma~\ref{lem:CharValOfU} this implies that the set $\{1 \leq i \leq e \ | \ \mu(\zeta_i,u,\theta) = \max(u) \}$ has at least two elements. Let $0 \subseteq M_1 \subseteq ... \subseteq M_{e+1} = M$ as in Corollary~\ref{cor:FlorianTamelyRamfied} with 
$M_i/M_{i-1}\cong  I_m^{\max(u)}$ for at least two values of $i$. Using Lemma~\ref{lem:ModSwap} we can assume without loss of generality that  $M_1 \cong I_m^{\max(u)}$ and $M_2/M_1 \cong I_m^{\max(u)}$. Assume by way of contradiction that $\gamma_{m+1}(M) < \max(u)$, which also implies $\gamma_{m+1}(M_2) < \max(u)$. Because $\gamma_m(M_1) = \gamma_m(M) = \max(u)$, the last equation coming from Lemma~\ref{lem:MinMax}, also $\gamma_m(M_2) = \max(u)$. So when we consider a skew tableau of $M_2/M_1$ the $\max(u)$-th row is empty. We divide this skew tableau by a vertical line into two skew tableaux such that the first $m-1$ columns lie on the left and the skew tableau on the right is called $T'$, cf. Figure~\ref{fig:NonTrivUNeg}. As $\gamma_m(M_1) = \gamma_m(M_2) = \max(u)$, we know that $T'$ has at most $\max(u)-1$ non-empty rows. This directly implies that $\gamma_1(T') < \max(u)$ as the maximal entry which can appear in $T'$ is $\max(u)-1$. Lemma~\ref{lem:divided_tableau} then gives $\gamma_m(M_2/M_1) < \max(u)$, contradicting the isomorphism type of $M_2/M_1$.

\[
\begin{tikzpicture}
%\draw  (0,2) -- (0,-1) -- (1,-1) -- (1,0) -- (3,0) -- (3,1) -- (4,1) -- (4,2) -- (5,2) -- (5,4) -- (4,4) -- (4,3) -- (1,3) -- (1,2) -- (0,2);
\draw  (0,2) -- (0,0) -- (1,0) -- (1,1) -- (1.5,1) -- (1.5,1.5) -- (2.5,1.5) -- (2.5,2) -- (0,2);
\draw  (3,2.5) -- (3.5,2.5) -- (3.5,3.5) -- (4.5,3.5) -- (4.5,4) -- (3,4) -- (3,2.5);

%horizontal lines
\draw[red, dashed] (-0.5,2) -- (5.5, 2);
\draw[red, dashed] (-0.5,2.5) -- (5.5, 2.5);

%vertical lines 
\draw[line width=0.1cm, gray] (2.5,4) -- (2.5, -1);

%labels
\node[label=west:{\Small{$\max(u)$} }] at (-0.5, 2) {};
\node[label=east:{\Small{$\max(u)-1$} }] at (5.5, 2.5) {};
\node[label=north:{\Small{$m-1$} }] at (2.5,4) {};

\end{tikzpicture}
\]

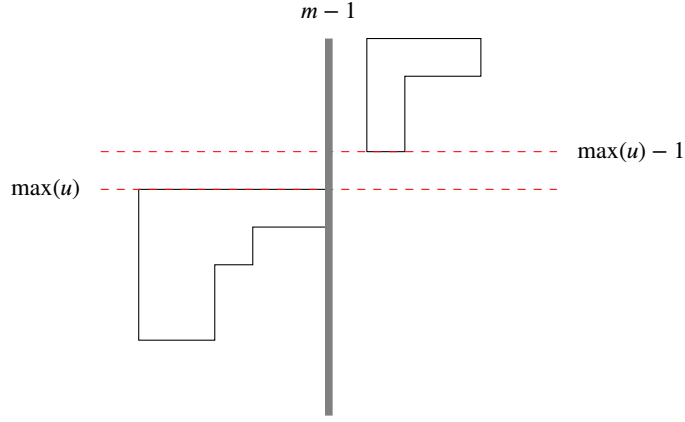
\captionof{figure}{Illustrating the case $\delta_\theta = -1$ in the proof of Lemma~\ref{lem:difference2Case}}\label{fig:NonTrivUNeg}
\vspace*{.5cm}

Now assume $\delta_\theta=1$. From the partial augmentations and Lemma~\ref{lem:CharValOfU} we obtain that the set $\{1 \leq i \leq e \ | \ \mu(\zeta_i,u,\theta) = \min(u)\}$ contains at least two elements. Let $0 \subseteq M_0 \subseteq \ldots \subseteq M_e = M$ be a filtration as in Corollary~\ref{cor:FlorianTamelyRamfied}. Using Lemma~\ref{lem:ModSwap} we can make it so that $M_{e}/M_{e-1} \cong M_{e-1}/M_{e-2} \cong I_m^{\min(u)}$ and $M_0$ is a direct sum of trivial modules. Now assume that $\gamma_{p-m}(M) > \min(u)$. Note that $\gamma_{p-m+1}(M) = \min(u)$ by Lemma~\ref{lem:MinMax}. Using Lemma~\ref{lem:ModSwap} we can write $M_{e-1}$ as a quotient $M/I_m^{\min(u)}$. Let us consider a skew tableau corresponding to this quotient. We see that the northwest corner of this tableau has shape as described in Lemma~\ref{lem:RectangelAndLine} with a rectangle of width $m$ and height $\min(u)$ removed, the first $\min(u)$ rows of $M$ containing at least $p-m+1$ boxes and the $(min(u)+1)$-th row containing $p-m$ boxes. The last fact follows from the assumption $\gamma_{p-m}(M) > \min(u)$. So by Lemma~\ref{lem:RectangelAndLine} we have $\gamma_{p+1-2m}(M_{e-1}) > \min(u)$ or $\gamma_{p+1-m}(M_{e-1}) > 0$. As $p+1-m=2+m(e-1) $ the inequality $\gamma_{p+1-m}(M_{e-1}) > 0$ is impossible due to Lemma~\ref{lem:TooBigIndSums}. We conclude $\gamma_{p+1-2m}(M_{e-1}) > \min(u)$. If we now consider the skew tableau of $M_{e-2}$ viewed as quotient $M_{e-1}/I_m^{\min(u)}$ (again using Lemma~\ref{lem:ModSwap}), then the $(\min(u)+1)$-th row has at least $p+1-2m$ boxes. By Lemma~\ref{lem:full_rectangle} this implies that $\gamma_{p+1-2m}(M_{e-2}) > 0$, which contradicts Lemma~\ref{lem:TooBigIndSums} as $p+1-2m=2+m(e-2) $.
\end{proof}

We are now in a position to show that the unit  $u$ is rationally conjugate to a trivial unit, which will complete the proof of Theorem~\ref{thm:main}.

\begin{proof}[Proof of Theorem~\ref{thm:main}]
Assume $u$ is not rationally conjugate to a trivial unit. Consider first the case that $\delta_\theta = -1$. By Lemma~\ref{lem:difference2Case} and Lemma~\ref{lem:MinMax} we know $\gamma_{m+1}(M) \geq \max(u)=1 + \min(g)$. On the other hand, Proposition~\ref{prop:MultSumIneq} and Lemma~\ref{lem:SumPsi1} imply $\gamma_{m+1}(M) \leq \sum_{\psi \in S}\delta_\psi \mu(\zeta,u,\psi)= \min(g)$. This is a contradiction, which shows that $u$ does not exist.

Now say $\delta_\theta = 1$. By Lemma~\ref{lem:difference2Case} and Lemma~\ref{lem:MinMax} we have $\gamma_{p-m}(M) \leq \min(u)=-1+\max(g)$. By Proposition~\ref{prop:MultSumIneq} and Lemma~\ref{lem:SumPsi1} we have $\gamma_{p-m}(M) \geq -\sum_{\psi \in S} \delta_\psi \mu(\zeta, u, \psi)=\max(g)$, which is again a contradiction. 
\end{proof}

\subsection{Remarks}

The following is a special case of our theorem, but it is still interesting to note that it is basically a direct corollary of known results.

\begin{proposition}
Assume $G$ is a $p$-solvable group and the Sylow $p$-subgroup of $G$ has order $p$. Then elements of order $p$ in $\V(\mathbb{Z}G)$ are rationally conjugate to elements of $G$.
\end{proposition}

\begin{proof}
Assume that $G$ is a minimal counterexample and $u \in \V(\mathbb{Z}G)$ has order $p$ and is not rationally conjugate to an element of $G$, i.e. $u$ has non-trivial partial augmentations by Theorem~\ref{th:MRSW}. If $N$ is a normal subgroup of $G$ of order not divisible by $p$, then the partial augmentations of $u$ in $\mathbb{Z}(G/N)$ do not become trivial by \cite[Lemma 2.1]{DokuchaevJuriaans}. So the only non-trivial minimal normal subgroup $N$ of $G$ is the Sylow $p$-subgroup. But then $u$ is rationally conjugate to an element of $G$ by \cite[Proposition 4.2]{Hertweck2006}.
\end{proof}

\section{Projective special linear groups}

There are direct applications of Theorem~\ref{thm:main} to units of prime order, which are often not achievable using other known methods. For instance, it follows that for $G$ the Mathieu group of degree $12$, all elements of order $11$ in $\V(\Z G)$ are rationally conjugate to trivial units. This was one of the cases studied in \cite{BovdiKonovalovM12} where the known methods proved insufficient. Another application would be that for the Suzuki group $G={\rm Sz(8)}$ units of order $7$ in $\V(\mathbb{Z}G)$ are rationally conjugate to trivial units, a case studied in \cite{BaechleDiplom}.
Many more applications of this type are possible, in particular for sporadic simple groups. 

But the methods that went into Theorem~\ref{thm:main} apply equally to units whose order is not prime. We show below that significant new results on the Zassenhaus conjecture for groups of type $\operatorname{PSL}(2,q)$ can be obtained using our methods. This is the type of simple groups for which $\V(\Z G)$  has been studied the most \cite{LutharPassi1989, WagnerDiplom, Bleher95, BleherHissKimmerle95, HertweckBrauer, HertweckA6, HertweckHoefertKimmerle, Gildea13, BaechleMargolisPSL2p3,MargolisPSL, delRioSerrano17, BachleMargolis4primaryI, BachleMargolis4primaryII, MargolisdelRioSerrano}. In fact, all non-abelian simple groups for which the Zassenhaus conjecture is known to hold are of the form $\PSL(2,q)$, and we prove it for an additional, potentially infinite family of values of $q$. While our methods apply to other types of simple groups for units of certain orders, we are not aware of any other examples where they would be sufficient to settle the Zassenhaus conjecture wholesale. 

Essentially all the pre-existing results mentioned above were obtained using the ``HeLP-method'', which was developed by Hertweck in \cite{HertweckBrauer} building on ideas by Luthar and Passi. It was implemented as a \texttt{GAP}-package in \cite{BachleMargolisHeLPPackage}. It is a character theoretic method which produces constraints on the partial augmentations of units of $p'$-order in $\V(\Z G)$ in terms of the $p$-modular representation theory of $G$. The lattice method enables us to use the $p$-modular representation theory of $G$ to produce contraints on units of order divisible by $p$ as well.
Our contribution here is in a case where the HeLP-method is known to be insufficent \cite{delRioSerrano17}.

\begin{theorem}\label{th:PSL2Ord2t}
Let $q$ be a power of an odd prime $p$, $G = \operatorname{PSL}(2,q)$ and $t$ an odd prime dividing $|G|$ such that $t^2$ does not divide $|G|$. Then units of order $2t$ in $\V(\mathbb{Z}G)$ are rationally conjugate to trivial units.
\end{theorem}

\begin{proof}
If $p = t$, then the there are no units of order $2t$ in $\V(\mathbb{Z}G)$ by \cite[Proposition 6.3]{HertweckBrauer}. So assume $p \neq t$. By \cite[Proposition 6.7]{HertweckBrauer} we may assume that $G$ contains elements of order $2t$ which is equivalent to $4t \equiv \pm 1 \bmod q$. Moreover, by \cite[Proposition 6.6]{HertweckBrauer} the result holds for $t=3$, so that we can assume $t \geq 5$. These kinds of units were studied in detail using the HeLP-method by del R\'io and Serrano who were able to strongly restrict the possible partial augmentations of these units, but not to show that all of them are rationally conjugate to trivial units \cite{delRioSerrano17}. They also showed that this case cannot be excluded using only the HeLP-method. We will handle this last case which remained in their paper.

We refer to \cite[Chapter II, Sections 6-8]{HuppertI} for the group theoretical properties of $G$ which we will use. The order of an element $g \in G$ is a divisor of $p, \frac{q-1}{2}$ or $\frac{q+1}{2}$. If $g$ is not of order $p$, then the only other element in the cyclic group $\langle g \rangle$ to which it is conjugate is $g^{-1}$. It follows that with respect to the prime $t$ the principal block of $G$ contains $\frac{t-1}{2}$ exceptional characters and two non-exceptional characters. The character table of $G$ was computed generically at the beginning of the 20th century independently by Jordan and Schur. We refer to \cite[Chapter 38]{Dornhoff} for an account. The character table is enough to read off the decomposition matrices of the blocks we will be interested in, but they are also explicitly given in \cite{Burkhardt}. It follows that the two non-exceptional characters in the principal $t$-block are the trivial character and the Steinberg character which is the unique complex irreducible character of degree $q$. We denote the Steinberg character by $\alpha$. The Brauer tree of the principal $t$-block is a line with three vertices.

Now assume $u \in \V(\mathbb{Z}G)$ has order $2t$ and is not rationally conjugate to a trivial unit. Then by \cite[Theorem 2.3]{delRioSerrano17} there is an element $g_0 \in G$ of order $2t$ such that $u^t$ is rationally conjugate to $g_0^t$, $u^2$ is rationally conjugate to $g_0^2$ and the only non-vanishing partial augmentations of $u$ are
\begin{equation}\label{eq:ParAugsPSL2}
\varepsilon_{g_0^{\frac{t-1}{2}}}(u) = 1, \ \ \varepsilon_{g_0^{\frac{t+1}{2}}}(u) = 1, \ \ \varepsilon_{g_0^{t-1}}(u) = -1. 
\end{equation}

We denote by $\zeta$ a fixed primitive complex $t$-th root of unity. $\mathcal{O}$ will denote the ring obtained by adjoining to $\mathbb{Z}_t$ a $t'$-root of unity such that an irreducible representation corresponding to a complex irreducible character $\chi$ can be realized over $\mathcal{O}[\chi]$. Such an $\mathcal{O}$ exists due to a theorem of Fong, cf. \cite[Remark 6]{BachleMargolis4primaryII}. Moreover, $k$ denotes the residue field of $\mathcal{O}$ and $M_\chi$ a $kG$-module obtained by realizing an irreducible complex character $\chi$ of $G$ over $\mathcal{O}[\chi]$ and reducing modulo the maximal ideal of $\mathcal{O}$. 
%We will sometimes write $\zeta^{\pm i}$ for the eigenvalues $\zeta^i$ and $\zeta^{-i}$ for some integer $i$.

We assume $u \in \V(\mathbb{Z}G)$ is a unit of order $2t$ having the partial augmentations given in \eqref{eq:ParAugsPSL2}. This will lead us to a contradiction, proving the theorem. 
Denote by ${\rm Tr}$ the map ${\rm Tr}_{\mathbb{Q}(\zeta)/\mathbb{Q}}$. For an ordinary character $\chi$ of $G$ and $i$ any integer the Luthar-Passi formula from Proposition~\ref{pr:luthar-passi-multiplicity-formula}, by the properties of $u$ given above, becomes
\begin{equation}\label{eq:MultPSL2}
\mu(-\zeta^i,u,\chi) = \frac{1}{2t}\left(\chi(1) - \chi(g_0^t) + {\rm Tr}\left(\chi(g_0^2)\zeta^{-2i}\right)  - {\rm Tr}\left(\left(\chi(g_0^{\frac{t-1}{2}}) + \chi(g_0^{\frac{t+1}{2}}) - \chi(g_0^{t-1})\right) \zeta^{-i} \right) \right)
\end{equation}
In case there is a rational constant $c$ such that $\chi(g) = c$ for any non-trivial element of order dividing $2t$ this is
\begin{equation}\label{eq:MultPSL2pRational}
\mu(-\zeta^i,u,\chi) = \frac{1}{2t}\left(\chi(1) - c + {\rm Tr}\left(c\zeta^{2i}\right)  - {\rm Tr}\left(c \zeta^{-i} \right)\right) = \frac{\chi(1)-c}{2t} 
\end{equation}

We will separate two cases corresponding to the value of $q$ modulo $4$. A module of type $M_\chi$ will always be considered as a $k\langle u \rangle$-module.

\paragraph* {\textbf{Case 1:} $q \equiv 1 \bmod 4$}
In this case the exceptional characters have degree $q+1$, i.e. the exceptional vertex of the Brauer tree is located in the middle, cf. Figure~\ref{fig:BrauerTreeMiddle}.

\[
\begin{tikzpicture}
\node[label=north:{$\mathbf{1}$}] at (0,1.5) (1){};
\node[label=north:{$\eta$}] at (1.5,1.5) (2){};
\node[label=north:{$\alpha$}] at (3,1.5) (3){};
\foreach \p in {1,2,3}{
\draw[fill=white] (\p) circle (.075cm);
}
\draw[fill=black] (2) circle (.075cm);
\draw (.075,1.5)--(1.425,1.5);
\draw (1.575,1.5)--(2.925,1.5);
\node[label=south:{$\frac{t-1}{2}$}] at (1.5,1.5) (2){};
\end{tikzpicture}
\] 

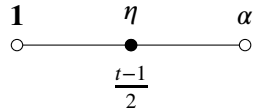
\captionof{figure}{Brauer tree of the principal block in case $q \equiv 1 \bmod 4$}\label{fig:BrauerTreeMiddle}
\vspace*{.5cm}

 Moreover, as there are elements of order $2t$ in $G$, we have $q \equiv 1 \bmod t$. The character $\alpha$ satisfies $\alpha(g) = 1$ for $g$ a non-trivial element of $G$ of order dividing $2t$.
%, so that
%\[\alpha(u^t) = \alpha(u^{t+1}) = \alpha(u) = 1. \]
%This leads to 
%%\[\mu(1,u^t,\alpha) = \frac{q+1}{2}, \ \mu(-1,u^t,\alpha) = \frac{q-1}{2}, \ \mu(1, u^{t+1}, \alpha) = \frac{q+t-1}{t}, \ mu(\zeta^i, u^{t+1}, \alpha) = \frac{q-1}{t},\]
%\[D_\alpha(u^t) \sim \left(1, \eigbox{\frac{q-1}{2}}{1,-1} \right), \ \ D_\alpha(u^{t+1}) \sim \left(1, \eigbox{\frac{q-1}{t}}{1,\zeta,\zeta^2,...,\zeta^{t-1}} \right)\]
%%where $i$ is any integer satisfying $1 \leq i \leq t-1$. 
%This in turn, together with the condition $\alpha(u) = 1$ leads to
%\begin{equation}\label{eq:MultsAlphaCase1}
%D_\alpha(u) \sim \left(1, \eigbox{\frac{q-1}{2t}}{1,\zeta,\zeta^2,...,\zeta^{t-1}}, \eigbox{\frac{q-1}{2t}}{-1,-\zeta,-\zeta^2,...,-\zeta^{t-1}} \right)
%\end{equation}
By \eqref{eq:MultPSL2pRational} we get $\mu(-\zeta,u,\alpha) = \frac{q-1}{2t}$ which gives 
\begin{equation}\label{eq:PSL2Case1pRational}
\gamma_{t-1,-1}(M_\alpha) = \frac{q-1}{2t}, 
\end{equation}
	by Corollary~\ref{cor:UnramifiedCase} as $\alpha$ is $t$-rational. Now pick an exceptional character $\eta$ such that $\eta(g_0^i) = \zeta^i + \zeta^{-i}$ for any integer $i$ not divisible by $2t$. Such a character exists in the principal block by \cite[III]{Burkhardt}. From \eqref{eq:MultPSL2} we observe
\begin{align*}
\mu (-\zeta^{\frac{t-1}{2}},u,\eta) &= \frac{1}{2t}\left(q+1 - 2 + {\rm Tr}\left((\zeta^2 + \zeta^{-2})\zeta^{-(t-1)}\right) \right. \\
 &-  {\rm Tr}\left(\left(\zeta^{\frac{t-1}{2}} + \zeta^{-\frac{t-1}{2}} + \zeta^{\frac{t+1}{2}} + \zeta^{-\frac{t+1}{2}} - \zeta^{t-1} - \zeta^{-(t-1)} \right)	 \zeta^{-\frac{t-1}{2}}\right)  \\
 &= \frac{1}{2t}(q-1 -2 - 2(t-1)) \\
 &= \frac{q-1}{2t} -1
\end{align*}
%Hence from the partial augmentations \eqref{eq:ParAugsPSL2} we obtain
%\[\eta(u^t) = 2, \ \  \eta(u^{t+1}) = \zeta + \zeta^{-1}, \ \ \eta(u) = \zeta^{\frac{t-1}{2}} + \zeta^{-\frac{t-1}{2}} + \zeta^{\frac{t+1}{2}} + \zeta^{-\frac{t+1}{2}} - (\zeta+\zeta^{-1}) = 2(\zeta^{\frac{t-1}{2}} + \zeta^{-\frac{t-1}{2}}) - (\zeta + \zeta^{-1}). \]
%These character values lead to
%\[D_\eta(u^t) \sim \left(1,1, \eigbox{\frac{q-1}{2}}{1,-1} \right), \ \ D_\eta(u^{t+1}) \sim \left(\zeta^{\pm 1}, \eigbox{\frac{q-1}{t}}{1,\zeta,\zeta^2,...,\zeta^{t-1}} \right)\]
%and
%\begin{align}\label{eq:MultsEtaCase1}
%D_\eta(u) \sim &\left(\zeta^{\pm \frac{t-1}{2}}, \eigbox{\frac{q-1}{2t}}{1,\zeta,\zeta^2,...,\zeta^{t-1}}, \right. \\
% &\left. -\zeta^{\pm 1}, -1,-\zeta^{\pm 1}, -\zeta^{\pm 2},...,-\zeta^{\pm \frac{t-3}{2}}, \eigbox{\left(\frac{q-1}{2t}-1\right)}{-1,-\zeta,-\zeta^2,...,-\zeta^{t-1}} \right) \nonumber
%\end{align}
So by Corollary~\ref{cor:FlorianTamelyRamfied} and Lemma~\ref{lem:CombOnM} applied to the part of $M_\eta$ on which $u^t$ acts as $-1$ we obtain $\gamma_{t-1,-1}(M_\eta) \leq \frac{q-1}{2t}-1$.  This together with \eqref{eq:PSL2Case1pRational} contradicts the fact that $M_\alpha$ is isomorphic to a submodule of $M_\eta$, a fact which follows from the shape of the Brauer tree.

%We conclude that $\gamma_{p-1,-1}(M_\alpha) = \frac{q-1}{2t}$, due to \eqref{eq:MultsAlphaCase1}. On the other hand applying Theorem~\ref{th:FlorianTamelyRamfied} and Lemma~\ref{lem:CombOnM} to the part of $M_\eta$ on which $u^p$ acts as $-1$ we obtain $\gamma_{p-1,-1}(M_\eta) \leq \left(\frac{q-1}{2t}-1\right)$, as $-\zeta^{\frac{t-1}{2}}$ appears only this amount of times as an eigenvalue of $D_\eta$. This contradicts the fact that $M_\alpha$ is isomorphic to a submodule of $M_\eta$, a fact which follows from the form of the decomposition matrix of the principal block.

\paragraph*{\textbf{Case 2:} $q \equiv -1 \bmod 4$}
Note that now $q \equiv -1 \bmod t$, as $G$ contains elements of order $2t$. The degree of the exceptional character is $q-1$ and the exceptional vertex of the Brauer tree is a leaf, cf.Figure~\ref{fig:BrauerTreeLeaf}.

\[
\begin{tikzpicture}
\node[label=north:{$\mathbf{1}$}] at (0,1.5) (1){};
\node[label=north:{$\alpha$}] at (1.5,1.5) (2){};
\node[label=north:{$\eta$}] at (3,1.5) (3){};
\foreach \p in {1,2,3}{
\draw[fill=white] (\p) circle (.075cm);
}
\draw[fill=black] (3) circle (.075cm);
\draw (.075,1.5)--(1.425,1.5);
\draw (1.575,1.5)--(2.925,1.5);
\node[label=south:{$\frac{t-1}{2}$}] at (3,1.5) (2){};
\end{tikzpicture}
\] 

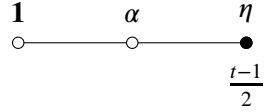
\captionof{figure}{Brauer tree of the principal block in case $q \equiv -1 \bmod 4$}\label{fig:BrauerTreeLeaf}
\vspace*{.5cm}

We have $\alpha(g) = -1$ for any non-trivial element of order dividing $2t$ which leads to $\mu(-\zeta,u,\alpha) = \frac{q+1}{2t}$ by \eqref{eq:MultPSL2pRational} and so
\begin{equation}\label{eq:PSL2Case2pRational}
\gamma_{2,-1}(M_\alpha) = \frac{q+1}{2t},
\end{equation}
by Corollary~\ref{cor:UnramifiedCase} as $\alpha$ is $t$-rational.
%\[\alpha(u^t) = \alpha(u^{t+1}) = \alpha(u) = -1, \]
%so that
%\[D_\alpha(u^t) \sim \left(-1, \eigbox{\frac{q-1}{2}}{1,-1} \right), \ \ D_\alpha(u^{t+1}) \sim \left(\zeta,\zeta^2,...,\zeta^{t-1}, \eigbox{\left(\frac{q+1}{t}-1\right)}{1,\zeta,\zeta^2,...,\zeta^{t-1}} \right)
% \]
% and 
%\begin{equation}\label{eq:MultsAlphaCase2}
%D_\alpha(u) \sim \left(\zeta,\zeta^2,...,\zeta^{-1}, \eigbox{\left(\frac{q+1}{2t}-1\right)}{1,\zeta,\zeta^2,...,\zeta^{t-1}}, \eigbox{\frac{q+1}{2t}}{-1,-\zeta,-\zeta^2,...,-\zeta^{t-1}} \right)
%\end{equation}
We choose an exceptional character $\eta$ such that $\eta(g_0^i) = -(\zeta^i + \zeta^{-i})$ for any integer $i$ not divisible by $2t$ which exists by \cite[VI]{Burkhardt}. We obtain by \eqref{eq:MultPSL2}
\begin{align*}
\mu (-\zeta^{\frac{t-1}{2}},u,\eta) &= \frac{1}{2t}\left(q-1 + 2 - {\rm Tr}\left((\zeta^2 + \zeta^{-2})\zeta^{-(t-1)}\right) \right. \\
 & +  {\rm Tr}\left((\zeta^{\frac{t-1}{2}} + \zeta^{-\frac{t-1}{2}} + \zeta^{\frac{t+1}{2}} + \zeta^{-\frac{t+1}{2}} - \zeta^{t-1} - \zeta^{-(t-1)} ) \zeta^{-\frac{t-1}{2}}\right)  \\
 &= \frac{1}{2t}(q+1 +2 + 2(t-1)) \\
 &= \frac{q+1}{2t} + 1
\end{align*}
%\[\eta(u^t) = -2, \ \  \eta(u^{t+1}) = -(\zeta + \zeta^{-1}), \ \ \eta(u) = -2(\zeta^{\frac{t-1}{2}} + \zeta^{-\frac{t-1}{2}}) + (\zeta + \zeta^{-1}). \]
%Hence,
%\[D_\eta(u^t) \sim \left(-1,-1, \eigbox{\frac{q-3}{2}}{1,-1} \right), \ \ D_\eta(u^{t+1}) \sim \left(1,\zeta^{\pm 2}, \zeta^{\pm 3},...,\zeta^{\pm \frac{t-1}{2}}, \eigbox{\left(\frac{q+1}{t}-1\right)}{1,\zeta,\zeta^2,...,\zeta^{t-1}} \right)\]
%and
%\begin{align}\label{eq:MultsEtaCase2}
%D_\eta(u) \sim &\left(1,\zeta^{\pm 1},\zeta^{\pm 2},...,\zeta^{\pm \frac{t-3}{2}}, \eigbox{\left(\frac{q+1}{2t}-1\right)}{1,\zeta,\zeta^2,...,\zeta^{t-1}}, \right. \\
% &\left. -\zeta^{\pm \frac{t-1}{2}}, -1,-\zeta^{\pm 2}, -\zeta^{\pm 3},...,-\zeta^{\pm \frac{t-1}{2}}, \eigbox{\left(\frac{q+1}{2t}-1\right)}{-1,-\zeta,-\zeta^2,...,-\zeta^{t-1}} \right) \nonumber
%\end{align}
%By \eqref{eq:MultsAlphaCase2} we have $\gamma_{2,-1}(M_\alpha) = \frac{q+1}{2t}$. As $\zeta^{\frac{t-1}{2}}$ appears $\left(\frac{q+1}{2t}+1 \right)$ times as an eigenvalue of $D_\eta(u)$ we moreover have $\gamma_{2,-1}(M_\eta) \geq \left(\frac{q+1}{2t}+1 \right)$ by Lemma~\ref{lem:CombOnM}, when we apply Theorem~\ref{th:FlorianTamelyRamfied} applied to the part of $M_\eta$ on which $u^p$ acts as $-1$. But this contradicts the fact that $M_\eta$ is isomorphic to a submodule of $M_\alpha$ and so $u$ can not exist.
Considering the part of $M_\eta$ on which $u^t$ acts as $-1$ this gives $\gamma_{2,-1}(M_\eta) \geq \frac{q+1}{2t}+1$ by Corollary~\ref{cor:FlorianTamelyRamfied} and Lemma~\ref{lem:CombOnM}. But this contradicts the fact that $M_\eta$ is isomorphic to a submodule of $M_\alpha$ by \eqref{eq:PSL2Case2pRational}.

\end{proof}

%\begin{corollary}
%Let $q$ be a prime or the square of a prime such that $q \equiv \pm 3 \mod 8$ and the odd part of $(q^2-1)(q^2+1)$ is squarefree. Then the Zassenhaus Conjecture holds for $\operatorname{PSL}(2,q)$.
%\end{corollary}
%\begin{proof}
%Let $G = \operatorname{PSL}(2,q)$ where $q=p$ or $q=p^2$ for a prime $p$ such that $p \equiv \pm 3 \mod q$. The orders of elements in $G$ are exactly the divisors of $p$, $\frac{q-1}{2}$ and $\frac{q+1}{2}$. Note that this implies that the order of a group element is not divisible by $4$. Let $u \in V(\mathbb{Z}G)$ be a unit of finite order. By \cite[Propositions 6.3, 6.7]{HertweckBrauer} and \cite[Theorem A]{BaechleMargolis4primaryI} the order of $u$ coincides with the order of some element of $G$. If the order of $u$ is $p$ it is rationally conjugate to an element of $G$ by \cite[Proposition 6.1]{HertweckBrauer}. If the order of $u$ is odd and different from $p$, it is rationally conjugate to an element of $G$ by \cite[Theorem 1.1]{MargolisdelRioSerrano}. As $G$ contains only one conjugacy class of involution also normalized units of order $2$ are rationally conjugate to elements of $G$. So the last remaining case is that $u$ has order $2r$ for $r$ some odd squarefree number. so $u$ is rationally conjugate to an element of $G$ by Theorem~\ref{th:PSL2Ord2t}. \ChLeo{I wish! Da fehlt noch was...}
%\end{proof}

\begin{corollary}\label{cor:zassenhauspsl}
Let $p$ be a prime and set $q=p$ or $q=p^2$. Assume that there exists a prime $t$ so that $q-1 = 4t$ or $q+1= 4t$. Then the Zassenhaus Conjecture holds for $\operatorname{PSL}(2,q)$.
\end{corollary}
\begin{proof}
Let $G = \operatorname{PSL}(2,q)$. The orders of elements in $G$ are exactly the divisors of $p$, $\frac{q-1}{2}$ and $\frac{q+1}{2}$. If $t=2$, then $q=7$ or $q=9$ and the statement is known to hold \cite[Example 3.6]{Hertweck2006}, \cite{HertweckA6}. So we can assume that $t$ is odd. Let $u \in \V(\mathbb{Z}G)$ be a unit of finite order. By \cite[Propositions 6.3, 6.7]{HertweckBrauer} and \cite[Theorem A]{BachleMargolis4primaryI} the order of $u$ coincides with the order of some element of $G$. If the order of $u$ is $p$ it is rationally conjugate to an element of $G$ by \cite[Proposition 6.1]{HertweckBrauer}. If the order of $u$ is odd and different from $p$, it is rationally conjugate to an element of $G$ by \cite[Theorem 1.1]{MargolisdelRioSerrano}. As $G$ contains only one conjugacy class of involutions, also all normalized units of order $2$ are rationally conjugate to elements of $G$. So the last remaining case is that where $u$ has order $2t$ for $t$ an odd prime different form $p$ so that $t^2$ does not divide $|G|$. Hence, $u$ is rationally conjugate to an element of $G$ by Theorem~\ref{th:PSL2Ord2t}.
\end{proof}

This corollary can be applied for instance to show the Zassenahus Conjecture for $\operatorname{PSL}(2,q)$ for $q \in \{29, 43, 53, 67 \}$. It is part of Dickson's conjecture that there are in fact infinitely many primes $p$ such that also $4p+1$ or $4p-1$ is a prime, so that the corollary would then apply for infinitely many values of $q$. 

\begin{remark}
Theoretical as well as computational evidence suggests that using the HeLP-method together with Theorem~\ref{th:PSL2Ord2t} would be enough to prove the Zassenhaus Conjecture for $\operatorname{PSL}(2,q)$ where $q=p$ or $q=p^2$ for an odd prime $p$ such that $(q^2-1)$ is not divisible by the square of any odd prime. A simple adjustment of the proof of \cite[Theorem 5.1]{CaicedoMargolis} shows that there are infinitely many such primes.
% \ChLeo{falls man Ordnung 18 hinkriegt, ist der zahlentheoretische Satz auch direkt anwendbar}
\end{remark}

%\ChLeo{In case we can treat 9 as a prime in the sense of this paper, so that order 18 would be possible to handle, one gets by computer that (ZC) holds for $\operatorname{PSL}(2,q)$ for any prime or prime-square $q$ such that $q \leq 101$. The problematic case is $q=37$.}

\subsection{Local application of the ``reversal of the lattice method''}
We will give an example of how Theorem~\ref{thm:reversaloflatticemethod} can be applied to produce non-trivial units locally. Our example does not provide new counterexamples to the Zassenhaus conjecture, as the HeLP-method applied in a characteristic not dividing the order of the unit in question disproves its existence over $\Z$. Nevertheless this raises the hope that similar examples  might produce new counterexamples to the Zassenhaus conjecture. These would need to be non-solvable though, as the following proposition shows.

\begin{proposition}\label{prop:CliffWeissApp}
Let $G$ be solvable and let $p$ and $q$ be prime divisors of $|G|$ so that the Sylow $p$-subgroup as well as the Sylow $q$-subgroup have prime order. Then any $u \in \V(\mathbb{Z}G)$ of order $pq$ is conjugate in $\mathbb{Q}G$ to an element of $G$.
\end{proposition}
\begin{proof}
Let $G$ be a counterexample of minimal order and $u \in \V(\mathbb{Z}G)$ an element of order $pq$ not rationally conjugate to an element of $G$. If $N$ is a normal subgroup of $G$ of order not divisible by $p$ or $q$, then the partial augmentations of $u$ do not become trivial in the quotient $\mathbb{Z}(G/N)$ by \cite[Lemma 2.1]{DokuchaevJuriaans}, so that $G/N$ would be a smaller counterexample to our claim. Hence we can assume that $G$ has a normal subgroup of order $p$. Let $P$ be a Sylow $p$-subgroup of $G$ and $C = C_G(P) = P \times K$ for some $p'$-subgroup $K$ of $G$ which exists by the Schur-Zassenhaus theorem. By \cite{HertweckOrders} we know that $G$ contains elements of order $pq$, so that there is a Sylow $q$-subgroup $Q$ of $G$ contained in $K$. Assume $K$ contains a normal subgroup $N$ of order not divisible by $q$. As $P$ is a characteristic subgroup of $G$, so are $C$ and $K$ and hence $N$ is also normal in $G$, something we have already excluded. So $Q$ is in fact a normal subgroup of $G$ and so is $P \times Q$. It is well-known and follows e.g. also from Theorem~\ref{th:pAsofTorsionUnits} that $u$ maps to the identity under the map $\mathbb{Z}G \rightarrow \mathbb{Z}(G/P\times Q)$. But then it is known by a theorem of Cliff and Weiss, which is more explicitly stated for this situation in \cite[Corollary 5.4]{MargolisdelRioCW1}, that $u$ is rationally conjugate to an element of $G$. 
\end{proof}

We will now give an example were we can show the existence of a non-trivial normalized torsion unit locally at the primes dividing its order.

\begin{example}\label{ex:PSL216}
Let $G = \operatorname{PSL}(2,16)$, so $|G| = 2^4 \cdot 3 \cdot 5 \cdot 17$. In particular, the $p$-blocks of $G$ have defect $0$ or $1$ for each odd prime $p$. Note that the Schur indices of all irreducible complex characters are equal to $1$ \cite{ShahabiShojaei}, so that no non-commutative skew fields come up in the Wedderburn-Artin decomposition of $\mathbb{Q}G$, which allows us to apply Theorem~\ref{thm:reversaloflatticemethod} to all blocks.

We will show that there is $u \in \V(\mathbb{Z}_{(3,5)}G)$ of order $15$ with non-trivial integral partial augmentations, where $\Z_{(3,5)}=\Z_{(3)}\cap \Z_{(5)}$. We use the notation of \texttt{GAP} for the conjugacy classes of $G$, and we also number the irreducible ordinary characters of $G$ as in the \texttt{GAP} character table library. Our unit will have the following properties: $u^5$ is rationally conjugate to an element in the conjugacy class $3a$, $u^3$ is rationally conjugate to an element in the conjugacy class $5a$ and the only non-vanishing partial augmentations of $u$ are $\varepsilon_{15c}(u) = -1$, $\varepsilon_{15d}(u) =2$. It can be checked for instance using the HeLP-package \cite{BachleMargolisHeLPPackage} that such a unit exists in $\mathbb{Q}G$. 
%Any $17$-Brauer character lifts to an ordinary character, so that the $17$-Brauer characters do not provide new restrictions on the partial augmentations of $u$ and hence a unit with the properties of $u$ also exists in $\mathbb{Z}_{17}G$ by ... \ChLeo{ref to reverse lattice}.

It remains to show that such a unit exists also $3$-locally and $5$-locally. All information on $G$ we will use can easily be computed using \texttt{GAP} \cite{GAP4}, including the claimed eigenvalues of $u$. We will denote by $\zeta$ a fixed primitive $15$-th root of unity.

We first consider the prime $3$ and set $\mathcal{O} = \mathbb{Z}_3$. We set $\zeta^6 = \xi_1$ and $\xi_j = \xi_1^j$ to use the notation of Theorem~\ref{thm:reversaloflatticemethod}. The subscripts of $\xi$ will be considered modulo $5$. There are two blocks of defect $1$ in $\mathcal{O}G$: the principal block and a non-principal block which splits into two blocks in $\mathcal{O}[\zeta^3]G$. The principal block contains three irreducible ordinary characters, namely $\chi_1$, $\chi_{10}$ and $\chi_{11}$ and two irreducible Brauer characters $\psi_1$ and $\psi_{10}$. We can visualize the Brauer tree as follows.

\[
\begin{tikzpicture}
\node[label=north:{$\chi_1$}] at (0,1.5) (1){};
\node[label=north:{$\chi_{11}$}] at (1.5,1.5) (2){};
\node[label=north:{$\chi_{10}$}] at (3,1.5) (3){};
\foreach \p in {1,2,3}{
\draw[fill=white] (\p) circle (.075cm);
}
\draw (3) circle (.075cm);
\draw (.075,1.5)--(1.425,1.5);
\draw (1.575,1.5)--(2.925,1.5);
\node[label=south:{$\psi_1$}] at (.75,1.7) (2){};
\node[label=south:{$\psi_{10}$}] at (2.2,1.7) (2){};
\end{tikzpicture}
\] 

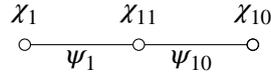
\captionof{figure}{Brauer tree of the principal $3$-block for $\operatorname{PSL}(2,16)$}
\vspace*{.5cm}

Note that there is no exceptional vertex. Clearly $\mu(1,u,\chi_1) = 1$ and $\mu(\eta, u, \chi_1)=0$ for any $\eta$ different from $1$. Moreover, we get
\begin{align*}
\mu(1,u, \chi_{10}) & = 2, \ \mu(\zeta^i,u,\chi_{10}) = 1 \ \ \text{for} \ \ 1 \leq i \leq 14 \\
\mu(\zeta^{\pm 5},u, \chi_{11}) & = 2, \ \mu(\zeta^i,u,\chi_{11}) = 1 \ \ \text{for} \ \ i \neq \pm 5, \  -7\leq i \leq 7 
\end{align*}
We conclude that we can choose, in the notation of Theorem~\ref{thm:reversaloflatticemethod}:
\begin{align*}
M_{\chi_1, 0} & = S_{\psi_1, 0} = I_1, \ M_{\chi_1, j} = S_{\psi_1,j} = 0 \ \ \text{for} \ \ j \not\equiv 0 \mod 5, \\
M_{\chi_{10}, 0} & = S_{\psi_{10}, 0} = I_3 \oplus I_1, \ M_{\chi_{10}, j} = S_{\psi_{10}, j} = I_3 \ \ \text{for} \ \ j \not\equiv 0 \mod 5, \\
M_{\chi_{11}, 0} & = I_3 \oplus I_2, \ M_{\chi_{11}, j} = I_3 \ \ \text{for} \ \ j \not\equiv 0 \mod 5.
\end{align*}
It is then easy to find filtrations as needed for the application of Theorem~\ref{thm:reversaloflatticemethod}.

We now consider the non-principal $3$-block $B$. To apply Theorem~\ref{thm:reversaloflatticemethod} it is sufficient to choose characters in ${\rm Irr}(B)$ and ${\rm IBr}(B)$ which are representatives of the action of ${\rm Gal}(\bar{\mathbb{Q}}_3/\mathbb{Q}_3)$ and ${\rm Gal}(\bar{\mathbb{F}}_3/\mathbb{F}_3)$ respectively. Such representatives are given by the ordinary characters $\chi_{12}$ and $\chi_{17}$ and the Brauer character $\psi_{11}$. Note that the block of $\mathcal{O}[\zeta^3]G$ containing $\chi_{17}$ also contains $\chi_{16}$, but $\chi_{16}$ and $\chi_{17}$ are Galois conjugate. The Brauer tree of the block of $\mathcal{O}[\zeta^3]G$ containing these characters looks as follows.

\[
\begin{tikzpicture}
\node[label=north:{$\chi_{12}$}] at (0,1.5) (1){};
\node[label=north:{$\chi_{16}, \chi_{17}$}] at (1.5,1.5) (2){};
\foreach \p in {1,2}{
\draw[fill=white] (\p) circle (.075cm);
}
\draw[fill=black] (2) circle (.075cm);
\node[label=south:{$\psi_{11}$}] at (.75,1.7) (2){};
\draw (.075,1.5)--(1.425,1.5);
\end{tikzpicture}
\] 

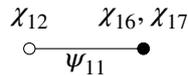
\captionof{figure}{Brauer tree of a non-principal $3$-block for $\operatorname{PSL}(2,16)$}\label{fig:BrauerTree216p3Principal}
\vspace*{.5cm}

We calculate
\begin{align*}
\mu(\zeta^{\pm 6},u, \chi_{12}) & = 2, \ \mu(\zeta^i,u,\chi_{12}) = 1 \ \ \text{for} \ \ i\neq \pm 6, \ -7 \leq i \leq 7 \\
\mu(\zeta^{\pm 1},u, \chi_{17}) & = 3, \ \mu(\zeta^{\pm 4},u, \chi_{17}) = 0, \ \mu(\zeta^i,u,\chi_{17}) = 1 \ \ \text{for} \ \ i\notin \{\pm 1, \pm 4 \}, \ -7 \leq i \leq 7
\end{align*}
Let 
\begin{align*}
M_{\chi_{12}, 0} & = M_{\chi_{17}, 0} = S_{\psi_{11}, 0} = I_3, \\
M_{\chi_{12}, \pm 2} & = M_{\chi_{17}, \pm 2} = S_{\psi_{11}, \pm 2} = I_3, \\
M_{\chi_{12}, \pm 1} & = M_{\chi_{17}, \pm 1} = S_{\psi_{11}, \pm 1} = I_2 \oplus I_1^2.
\end{align*}
It is again easy to choose filtrations satisfying the conditions of Theorem~\ref{thm:reversaloflatticemethod}. Note that it was important to choose $ S_{\psi_{11}, \pm 1}$ as $I_2 \oplus I_1^2$, though the eigenvalues of $\chi_{12}$ would allow to choose $I_3 \oplus I_1$, while the eigenvalues of $\chi_{17}$ would also allow to choose $I_1^4$, but none of these choices would be consistent with the other character. We conclude that $u$ is conjugate in $\mathbb{Q}_3G$ to an element of $\mathbb{Z}_3G$.

We next consider the prime $5$ and set $\OO=\Z_5$. We also set $\xi_1 = \zeta^{10}$, $\xi_2=\zeta^5$, $\xi_3=1$ and view the subscripts of $\xi$ modulo $3$. There are two blocks of defect $1$: the principal block and a non-principal block. The principal block contains four irreducible ordinary characters, namely $\chi_1$, $\chi_{10}$, $\chi_{12}$ and $\chi_{13}$ and two irreducible Brauer characters $\psi_1$ and $\psi_{10}$ The Brauer tree has the following shape.

\[
\begin{tikzpicture}
\node[label=north:{$\chi_1$}] at (0,1.5) (1){};
\node[label=north:{$\chi_{12}, \chi_{13}$}] at (1.5,1.5) (2){};
\node[label=north:{$\chi_{10}$}] at (3,1.5) (3){};
\foreach \p in {1,2,3}{
\draw[fill=white] (\p) circle (.075cm);
}
\draw[fill=black] (2) circle (.075cm);
\draw (.075,1.5)--(1.425,1.5);
\draw (1.575,1.5)--(2.925,1.5);
\node[label=south:{$\psi_1$}] at (.75,1.7) (2){};
\node[label=south:{$\psi_{10}$}] at (2.2,1.7) (2){};
%\node[label=south:{$2$}] at (3,1.5) (2){};
\end{tikzpicture}
\] 
\captionof{figure}{Brauer tree of the principal $5$-block for $\operatorname{PSL}(2,16)$}\label{fig:BrauerTree216p5Principal}
\vspace*{.5cm}
Note that the exceptional vertex lies in the middle and has multiplicity $2$. The multiplicities of the eigenvalues for $\chi_{10}$ and $\chi_{12}$ are given above, and as $\chi_{12}$ and $\chi_{13}$ are Galois conjugate, it suffices to know the multiplicities of eigenvalues for one of them. We now set
\begin{align*}
M_{\chi_1, 0} & = S_{\psi_1, 0} = I_1, \ M_{\chi_1,\pm 1} = S_{\psi_1, \pm 1} = 0, \\
M_{\chi_{10}, 0} & = S_{\psi_{10}, 0} = I_5 \oplus I_1, \ M_{\chi_{10}, \pm 1} = S_{\psi_{10}, \pm 1} = I_5, \\
M_{\chi_{12}, 0} & = I_5 \oplus I_2, \ M_{\chi_{12}, \pm 1} = I_5,
\end{align*}
which again allows filtrations as needed for the application of Theorem~\ref{thm:reversaloflatticemethod}.

Finally, we consider the non-principal $5$-block. This block contains the ordinary characters $\chi_{11}$, $\chi_{14}$, $\chi_{15}$, $\chi_{16}$, $\chi_{17}$ and the Brauer character $\psi_{11}$. The Brauer tree has shape

\[
\begin{tikzpicture}
\node[label=north:{$\chi_{11}$}] at (0,1.5) (1){};
\node[label=north:{$\chi_{14}, \chi_{15}, \chi_{16}, \chi_{17}$}] at (2.5,1.5) (2){};
\foreach \p in {1,2}{
\draw[fill=white] (\p) circle (.075cm);
}
\draw[fill=black] (2) circle (.075cm);
\node[label=south:{$\psi_{11}$}] at (1.2,1.7) (2){};
\draw (.075,1.5)--(2.425,1.5);
\end{tikzpicture}
\] 
\captionof{figure}{Brauer tree of the non-principal $5$-block for $\operatorname{PSL}(2,16)$}\label{fig:BrauerTree216p5NonPrincipal}
\vspace*{.5cm}
So the multiplicity of the exceptional vertex is $4$. We can choose any exceptional character as a representative of the Galois orbit and we choose $\chi_{17}$ as we have computed the multiplicities of the eigenvalues for this character already for the $3$-local case. We get that the following choices allow us to define the necessary filtrations:
\begin{align*}
M_{\chi_{11}, 0} & = M_{\chi_{17}, 0} = S_{\psi_{11}, 0} = I_5, \\
M_{\chi_{11}, \pm 1} & = M_{\chi_{17}, \pm 1} = S_{\psi_{11}, \pm 1} = I_4 \oplus I_1^2.
\end{align*}
We conclude that $u$ is indeed conjugate in $\mathbb{Q}_5G$ to an element in $\mathbb{Z}_5G$.

We have thus proven the existence of $u$ locally for the primes $3$ and $5$. By \cite[Proposition 3.2]{EiseleMargolis} a unit with the claimed properties exists in $\mathbb{Z}_{(3,5)}G$. It is known though that units of order $15$ in $\mathbb{Z}G$ are rationally conjugate to elements of $G$ and even that the Zassenhaus Conjecture holds for $G$ \cite[Theorem C]{BachleMargolis4primaryII}. The obstruction comes from the prime $2$, which does not divide the order of the unit.
\end{example}

\bibliographystyle{amsalpha}
\bibliography{ramlat}
\end{document}